\documentclass[final,leqno]{siamltex704}
\usepackage{amsmath}
\usepackage{amssymb}
\usepackage{graphicx}
\usepackage[notcite,notref]{showkeys}

\newtheorem{remark}{Remark}[section]

\newcommand{\bq}{{\bf q}}
\newcommand{\vq}{{\bf q}}
\newcommand{\bw}{{\bf w}}

\def\T{{\mathcal T}}
\def\E{{\mathcal E}}

\def\bn{{\bf n}}
\def\vn{{\bf n}}

\def\bq{{\bf q}}

\def\ljump{{[\![}}
\def\rjump{{]\!]}}

\def\3bar{{|\hspace{-.02in}|\hspace{-.02in}|}}

\newcommand{\ldp}{{(\hspace{-.03in}(}}
\newcommand{\rdp}{{)\hspace{-.03in})}}
\newcommand{\bP}{{\bf P}}
\newcommand{\bPi}{\boldsymbol{\Pi}}
\newcommand{\vchi}{\boldsymbol{\chi}}

\title{A weak Galerkin mixed finite element method for biharmonic equations }
\author{Lin Mu\thanks{Department of
Mathematics, Michigan State University, East Lancing, MI 48824
(lxmu@ualr.edu).}\and Junping Wang\thanks{Division of Mathematical
Sciences, National Science Foundation, Arlington, VA 22230
(jwang@\break nsf.gov). The research of Wang was supported by the
NSF IR/D program, while working at the Foundation. However, any
opinion, finding, and conclusions or recommendations expressed in
this material are those of the author and do not necessarily reflect
the views of the National Science Foundation.} \and Yanqiu
Wang\thanks{Department of Mathematics, Oklahoma State University,
Stillwater, OK 74075 (yqwang@math.okstate.edu).} \and Xiu
Ye\thanks{Department of Mathematics, University of Arkansas at
Little Rock, Little Rock, AR 72204 (xxye@ualr.edu). This research
was supported in part by National Science Foundation Grant
DMS-1115097.}}
\begin{document}
\maketitle

\begin{abstract}
This article introduces and analyzes a weak Galerkin mixed finite
element method for solving the biharmonic equation. The weak
Galerkin method, first introduced by two of the authors (J. Wang and
X. Ye) in \cite{WangYe_PrepSINUM_2011} for second order elliptic
problems, is based on the concept of {\em discrete weak gradients}.
The method allows the use of completely discrete finite element
functions on partitions of arbitrary polygon or polyhedron. In this
article, the weak Galerkin method is applied to discretize the
Ciarlet-Raviart mixed formulation for the biharmonic equation. In
particular, an a priori error estimation is given for the
corresponding finite element approximations. The error analysis
essentially follows the framework of Babu\u{s}ka, Osborn, and
Pitk\"{a}ranta \cite{Babuska80} and uses specially designed
mesh-dependent norms. The proof is technically tedious due to the
discontinuous nature of the weak Galerkin finite element functions.
Some computational results are presented to demonstrate the
efficiency of the method.
\end{abstract}

\begin{keywords}
Weak Galerkin finite element methods,  discrete gradient, biharmonic
equations, mixed finite element methods.
\end{keywords}

\begin{AMS}
Primary, 65N15, 65N30.
\end{AMS}
\pagestyle{myheadings}

\section{Introduction} In this paper, we are concerned with
numerical methods for the following biharmonic equation with clamped
boundary conditions
\begin{equation} \label{pde}
\begin{aligned}
\Delta^2 u&= f\qquad
\mbox{in}\;\Omega,\\
u&=0\qquad \mbox{on}\; \partial\Omega,\\
\frac{\partial u}{\partial\bn}&=0 \qquad \mbox{on}\; \partial\Omega,
\end{aligned}
\end{equation}
where $\Omega$ is a polygonal or polyhedral domain in
$\mathbb{R}^d\; (d=2,3)$. To solve the problem (\ref{pde}) using a
primal-based conforming finite element method, one would need $C^1$
continuous finite elements, which usually involve large degree of
freedoms and hence can be computationally expensive. There are
alternative numerical methods, for example, by using either
nonconforming elements \cite{Adini, Lascaux, Morley}, the $C^0$
discontinuous Galerkin method \cite{Engel02, BrennerSung},  or mixed
finite element methods \cite{Behrens11, bf, Ciarlet74, Cockburn09,
Gudi08, Herrmann67, Herrmann, Johnson73, Johnson82, Malkus78,
Miyoshi73}. One of the earliest mixed formulation proposed for
(\ref{pde}) is the Ciarlet-Raviart mixed finite element formulation
\cite{Ciarlet74} which decomposes (\ref{pde}) into a system of
second order partial differential equations. More precisely, in this
formulation, one introduces a dual variable $w=-\Delta u$ and
rewrites the four-order biharmonic equation into two coupled second
order equations
\begin{equation}\label{w-mix}
\begin{cases}
w + \Delta u = 0, \\
-\Delta w = f,
\end{cases}
\end{equation}
In \cite{Ciarlet74}, the above system of second order equations is
discretized by using the standard $H^1$ conforming elements.
However, only sub-optimal error estimates are proved in
\cite{Ciarlet74} for quadratic or higher order of elements. Improved
error estimates have been established in \cite{Babuska80, Falk80,
Glowinski79, Scholz78} for quadratic or higher order of elements. In
\cite{Babuska80}, Babu\u{s}ka, Osborn and Pitk\"{a}ranta pointed out
that a suitable choice of norms are $L^2$ for $w$ and $H^2$ for $u$,
or equivalent, in order to use the standard LBB stability analysis.
In this sense, one has ``optimal'' order of convergence in $H^2$
norm for $u$ and in $L^2$ norm for $w$, for quadratic or higher
order of elements. However, when equal order approximation is used
for both $u$ and $w$, the ``optimal'' order of error estimate is
restricted by the interpolation error in $H^2$ norm, and thus may
not be really optimal. Moreover, this standard technique does not
apply to the piecewise linear discretization, since in this case the
interpolation error can not even be measured in $H^2$ norm. A
solution to this has been proposed by Scholz \cite{Scholz78}. Using
an $L^{\infty}$ argument, Scholz was able to improve the convergence
rate in $L^2$ norm for $w$ by $h^{\frac12}$, and this theoretical
result is known to be sharp. Also, Scholz's proof works for all
equal-order elements including piecewise linears.

The goal of this paper is to propose and analyze a weak Galerkin
discretization method for the mixed formulation (\ref{w-mix}). The
weak Galerkin method was recently introduced in
\cite{WangYe_PrepSINUM_2011} for second order elliptic equations. It
is an extension of the standard Galerkin finite element method where
classical derivatives were substituted by weakly defined derivatives
on functions with discontinuity. Optimal order of a priori error
estimates has been observed and established for various weak
Galerkin discretization schemes for second order elliptic equations
\cite{WangYe_PrepSINUM_2011, wy-mixed, mwy-wg-stabilization}. A
numerical implementation of weak Galerkin was discussed in
\cite{MuWangWangYe, mwy-wg-stabilization} for some model problems.

Applying the weak Galerkin method to both second-order equations in
(\ref{w-mix}) appears to be trivial and straight-forward at first
glance. However, the application turns out to be much more
complicated than simply combining one weak Galerkin scheme with
another one. The application is particularly non-trivial in the
mathematical theory on error analysis. In deriving an a priori error
estimate, we follow the framework as developed in \cite{Babuska80}
by using mesh-dependent norms. Many commonly used properties and
inequalities for standard Galerkin finite element method need to be
re-derived for weak Galerkin methods with respect to the
mesh-dependent norms.  Due to the discrete nature of the weak
Galerkin functions, technical difficulties arise in the derivation
of inequalities or estimates. The technical estimates and tools that
we have developed in this paper should be essential to the analysis
of weak Galerkin methods for other type of modeling equations. They
should also play an important role in future developments of
preconditioning techniques for weak Galerkin methods. Therefore, we
believe this paper provides useful technical tools for future
research, in addition to introducing an efficient new method for
solving biharmonic equations.

The paper is organized as follows. In Section
\ref{sec:weakGalerkin}, a weak Galerkin discretization scheme for
the Ciarlet-Raviart mixed formulation of the biharmonic equation is
introduced and proved to be well-posed. Section \ref{sec:tools} is
dedicated to defining and analyzing several technical tools,
including projections, mesh-dependent norms and some estimates. With
the aid of these tools, an error analysis is presented in Section
\ref{sec:erroranalysis}. Finally, in Section \ref{sec:numerical}, we
report some numerical results that show the efficiency of the
method.

\section{A Weak Galerkin Finite Element Scheme} \label{sec:weakGalerkin}

For illustrative purpose, we consider only the two-dimensional case
of (\ref{pde}) and the corresponding weak Galerkin method will be
based on a shape-regular triangulation of the domain $\Omega$.

Let $D\subseteq\Omega$ be a polygon, we use the standard definition
of Sobolev spaces $H^s(D)$ and $H_0^s(D)$ with $s\ge 0$ (e.g., see
\cite{adams, ciarlet} for details). The associated inner product,
norm, and semi-norms in $H^s(D)$ are denoted by
$(\cdot,\cdot)_{s,D}$, $\|\cdot\|_{s,D}$, and $|\cdot|_{r,D}, 0\le r
\le s$, respectively. When $s=0$, $H^0(D)$ coincides with the space
of square integrable functions $L^2(D)$. In this case, the subscript
$s$ is suppressed from the notation of norm, semi-norm, and inner
products. Furthermore, the subscript $D$ is also suppressed when
$D=\Omega$. For $s<0$, the space $H^s(D)$ is defined to be the dual
of $H_0^{-s}(D)$.

Occasionally, we need to use the more general Sobolev space
$W^{s,p}(\Omega)$, for $1\le p\le \infty$, and its norm
$\|\cdot\|_{W^{s,p}(\Omega)}$. The definition simply follows the
standard one given in \cite{adams, ciarlet}. When $s=0$, the space
$W^{s,p}(\Omega)$ coincides with $L^p(\Omega)$.

The above definition/notation can easily be extended to
vector-valued and matrix-valued functions. The norm, semi-norms, and
inner-product for such functions shall follow the same naming
convention. In addition, all these definitions can be transferred
from a polygonal domain $D$ to an edge $e$, a domain with lower
dimension. Similar notation system will be employed. For example,
$\|\cdot\|_{s,e}$ and $\|\cdot\|_e$ would denote the norm in
$H^s(e)$ and $L^2(e)$ etc. We also define the $H(div)$ space as
follows
$$
H(div,\Omega) = \{\vq:\ \vq \in [L^2(\Omega)]^2,\: \nabla\cdot\vq\in
L^2(\Omega)\}.
$$

Using notations defined above, the variational form of the
Ciarlet-Raviart mixed formulation (\ref{w-mix}) seeks $u\in
H_0^1(\Omega)$ and $w\in H^1(\Omega)$ satisfying
\begin{equation}\label{mixed}
\begin{cases}
  (w, \phi) - (\nabla u, \nabla \phi) =  0\qquad &\textrm{for all } \phi\in H^1(\Omega), \\
  (\nabla w, \nabla\psi) = (f, \psi) \qquad &\textrm{for all } \psi\in H_0^1(\Omega).
\end{cases}
\end{equation}
For any solution $w$ and $u$ of (\ref{mixed}), it is not hard to see
that $w=-\Delta u$. In addition, by choosing $\phi=1$ in the first
equation of (\ref{mixed}), we obtain
$$
\int_{\Omega} w \, dx = 0.
$$
Define $\bar{H}^1(\Omega) \subset H^1(\Omega)$ by
$$
\bar{H}^1(\Omega) = \{v:\ v\in H^1(\Omega),\, \int_{\Omega} v\, dx =
0 \},
$$
which is a subspace of $H^1(\Omega)$ with mean-value free functions.
Clearly, the solution $w$ of (\ref{mixed}) is a function in
$\bar{H}^1(\Omega)$.

One important issue in the analysis is the regularity of the
solution $u$ and $w$. For two-dimensional polygonal domains, this
has been thoroughly discussed in \cite{blum80}. According to their
results, the biharmonic equation with clamped boundary condition
(\ref{pde}) satisfies
\begin{equation} \label{eq:regularity}
  \|u\|_{4-k} \le c\|f\|_{-k},
\end{equation}
where $c$ is a constant depending only on the domain $\Omega$. Here
the parameter $k$ is determined by
$$
\begin{aligned}
  k=1 \quad&\textrm{if all internal angles of } \Omega \textrm{ are less than } 180^\circ \\
  k=0 \quad&\textrm{if all internal angles of } \Omega \textrm{ are less than } 126.283696\cdots^\circ \\
\end{aligned}
$$
The above regularity result indicates that the solution $u\in
H^3(\Omega)$ when $\Omega$ is a convex polygon and $f\in
H^{-1}(\Omega)$. It follows that the auxiliary variable $w\in
H^1(\Omega)$. Moreover, if all internal angles of $\Omega$ are less
than $126.283696\cdots^\circ$ and $f\in L^2(\Omega)$, then $u\in
H^4(\Omega)$ and $w\in H^2(\Omega)$. The drawback of the mixed
formulation (\ref{mixed}) is that the auxiliary variable $w$ may not
possess the required regularity when the domain is non-convex. We
shall explore other weak Galerkin methods to deal with such cases.


Next, we present the weak Galerkin discretization of the
Ciarlet-Raviart mixed formulation. Let ${\cal T}_h$ be a
shape-regular, quasi-uniform triangular mesh on a polygonal domain
$\Omega$, with characteristic mesh size $h$. For each triangle $K\in
{\cal T}_h$, denote by $K_0$ and $\partial K$ the interior and the
boundary of $K$, respectively. Also denote by $h_K$ the size of the
element $K$. The boundary $\partial K$ consists of thee edges.
Denote by $\E_h$ the collection of all edges in ${\cal T}_h$. For
simplicity of notation, throughout the paper, we use ``$\lesssim$''
to denote ``less than or equal to up to a general constant
independent of the mesh size or functions appearing in the
inequality''.

Let $j$ be a non-negative integer. On each $K\in {\cal T}_h$, denote
by $P_j(K_0)$ the set of polynomials with degree less than or equal
to $j$. Likewise, on each $e\in \E_h$, $P_j(e)$ is the set of
polynomials of degree no more than $j$. Following
\cite{WangYe_PrepSINUM_2011}, we define a weak discrete space on
mesh $\T_h$ by
$$
V_{h} =  \{v:\:  v|_{K_0}\in P_j(K_0),\ K\in \T_h;
 \ v|_e\in P_j(e), e\in \E_h\}.
$$
Observe that the definition of $V_h$ does not require any continuity
of $v\in V_h$ across the interior edges. A function in $V_h$ is
characterized by its value on the interior of each element plus its
value on the edges/faces. Therefore, it is convenient to represent
functions in $V_h$ with two components, $v=\{v_0, v_b\}$, where
$v_0$ denotes the value of $v$ on all $K_0$ and $v_b$ denotes the
value of $v$ on $\E_h$.

We further define an $L^2$ projection from $H^1(\Omega)$ onto $V_h$
by setting $Q_h v \equiv \{Q_0 v,\, Q_b v\}$, where $Q_0 v|_{K_0}$
is the local $L^2$ projection of $v$ in $P_j(K_0)$, for $K\in\T_h$,
and $Q_b v|_e$ is the local $L^2$ projection in $P_j(e)$, for $e\in
\E_h$. To take care of the homogeneous Dirichlet boundary condition,
define
$$
V_{0,h} = \{v\in V_h\: :\: v=0\textrm{ on } \E_h\cap\partial\Omega\}.
$$
It is not hard to see that the $L^2$ projection $Q_h$ maps $H_0^1(\Omega)$ onto $V_{0,h}$.

The weak Galerkin method seeks an approximate
solution $[u_h; \, w_h]\in V_{0,h}\times V_h$ to the mixed form of
the biharmonic problem (\ref{w-mix}). To this end, we first
introduce a discrete $L^2$-equivalent inner-product and a discrete
gradient operator on $V_h$. For any $v_h=\{v_0, v_b\}$ and
$\phi_h=\{\phi_0, \phi_b\}$ in $V_h$, define an inner-product as
follows
$$
\ldp v_h, \phi_h\rdp \triangleq \sum_{K\in \T_h}(v_0, \phi_0)_K +  \sum_{K\in \T_h} h_K \langle v_0-v_b, \phi_0-\phi_b\rangle_{\partial K}.
$$
It is not hard to see that $\ldp v_h, v_h\rdp = 0$ implies
$v_h\equiv 0$. Hence, the inner-product is well-defined. Notice that
the inner-product $\ldp\cdot,\cdot\rdp$ is also well-defined for any
$v\in H^1(\Omega)$ for which $v_0=v$ and $v_b|_e=v|_e$ is the trace
of $v$ on the edge $e$. In this case, the inner-product
$\ldp\cdot,\cdot\rdp$ is identical to the standard $L^2$
inner-product.

The discrete gradient operator is defined element-wise on each $K\in
\T_h$. To this end, let $RT_j(K)$ be a space of Raviart-Thomas
element \cite{rt} of order $j$ on triangle $K$. That is,
$$
RT_j(K) = (P_j(K))^2 + \mathbf{x} P_j(K).
$$
The degrees of freedom of $RT_j(K)$ consist of moments of normal
components on each edge of $K$ up to order $j$, plus all the moments
in the triangle $K$ up to order $(j-1)$. Define
$$
\Sigma_h = \{\vq\in (L^2(\Omega))^2:\: \vq|_K \in RT_j(K),\ K\in
\T_h\}.
$$
Note that $\Sigma_h$ is not necessarily a subspace of
$H(div,\Omega)$, since it does not require any continuity in the
normal direction across any edge. A discrete weak gradient
\cite{WangYe_PrepSINUM_2011} of $v_h=\{v_0,v_b\}\in V_h$ is defined
to be a function $\nabla_w v_h \in \Sigma_h$ such that on each
$K\in\T_h$,
\begin{equation}\label{discrete-weak-gradient-new}
(\nabla_{d} v_h, \vq)_K = -(v_0, \nabla\cdot \vq)_K + \langle v_b, \vq\cdot\bn\rangle_{\partial K},
\quad \textrm{for all } \vq\in RT_j(K),
\end{equation}
where $\bn$ is the unit outward normal on $\partial K$. Clearly,
such a discrete weak gradient is always well-defined. Also, the
discrete weak gradient is a good approximation to the classical
gradient, as demonstrated in \cite{WangYe_PrepSINUM_2011}:

\medskip
\begin{lemma} \label{lem:assumptions}
  For any $v_h=\{v_0,\,v_b\}\in V_h$ and $K\in\T_h$, $\nabla_w v_h|_K = 0$ if and only if $v_0=v_b = constant$ on $K$.
Furthermore, for any $v\in H^{m+1}(\Omega)$, where $0\le m\le j+1$, we have
    $$\|\nabla_w (Q_h v)-\nabla v\|\lesssim h^{m}\|v\|_{m+1}.$$
\end{lemma}
\medskip

We are now in a position to present the weak Galerkin finite element
formulation for the biharmonic problem (\ref{w-mix}) in the mixed
form: {\it Find $u_h = \{u_0,\, u_b\}\in V_{0,h}$ and $w_h =
\{w_0,\, w_b\}\in V_h$ such that}
\begin{equation}\label{eq:wg}
\begin{cases}
\ldp w_h,\,\phi_h\rdp -(\nabla_w u_h,\, \nabla_w\phi_h)=0, \qquad &\textrm{for all }\phi_h = \{\phi_0,\, \phi_b\}\in V_h,\\
(\nabla_w w_h,\,\nabla_w\psi_h)=(f,\,\psi_0), \qquad &\textrm{for
all }\psi_h = \{\psi_0,\, \psi_b\}\in V_{0,h}.
\end{cases}
\end{equation}

\medskip
\begin{theorem}\label{unique}
The weak Galerkin finite element formulation (\ref{eq:wg}) has one
and only one solution $[u_h; w_h]$ in the corresponding finite
element spaces.
\end{theorem}

\begin{proof}
For the discrete problem arising from (\ref{eq:wg}), it suffices to
show that the solution to (\ref{eq:wg}) is trivial if $f=0$; the
existence of solution stems from its uniqueness.

Assume that $f=0$ in (\ref{eq:wg}). By taking $\phi_h=w_h$ and
$\psi_h=u_h$ in (\ref{eq:wg}) and adding the two resulting equations
together, we immediately have $\ldp w_h,\, w_h\rdp = 0$, which
implies $w_h\equiv 0$. Next, by setting $\phi_h=u_h$ in the first
equation of (\ref{eq:wg}), we arrive at $(\nabla_w u_h, \nabla_w
u_h)=0$. By using Lemma \ref{lem:assumptions}, we see that $u_h$
must be a constant in $\Omega$, which together with the fact that
$u_h=0$ on $\partial\Omega$ implies $u_h\equiv 0$ in $\Omega$. This
completes the proof of the theorem.
\end{proof}
\medskip

One important observation of (\ref{eq:wg}) is that the solution
$w_h$ has mean value zero over the domain $\Omega$, which is a
property that the exact solution $w=-\Delta u$ must possess. This
can be seen by setting $\phi_h=1$ in the first equation of
(\ref{eq:wg}), yielding
$$
(w_h, 1) = \ldp w_h, 1\rdp = (\nabla_w u_h,\nabla_w 1) = 0,
$$
where we have used the definition of $\ldp\cdot,\cdot\rdp$ and Lemma
\ref{lem:assumptions}. For convenience, we introduce a space
$\bar{V}_h \subset V_h$ defined as follows
$$
\bar{V}_h = \{v_h:\ v_h=\{v_0,v_b\}\in V_h,\, \int_{\Omega} v_0\, dx
= 0\}.
$$

\section{Technical Tools: Projections, Mesh-dependent Norms and Some Estimates} \label{sec:tools}
The goal of this section is to establish some technical results
useful for deriving an error estimate for the weak Galerkin finite
element method (\ref{eq:wg}).

\subsection{Some Projection Operators and Their Properties}
Let $\bP_h$ be the $L^2$ projection from $(L^2(\Omega))^2$ to
$\Sigma_h$, and $\bPi_h$ be the classical interpolation \cite{bf}
from $(H^{\gamma}(\Omega))^2, \gamma>\frac12,$ to $\Sigma_h$ defined
by using the degrees of freedom of $\Sigma_h$ in the usual mixed
finite element method. It follows from the definition of $\bPi_h$
that $\bPi_h \vq\in H(div,\Omega)\cap \Sigma_h$ for all $\vq\in
(H^{\gamma}(\Omega))^2$. In other words, $\bPi_h \vq$ has continuous
normal components across internal edges. It is also well-known that
$\bPi_h$ preserves the boundary condition $\vq\cdot\vn|_{\partial
\Omega} =0$, if it were imposed on $\vq$. The properties of $\bPi_h$
has been well-developed in the context of mixed finite element
methods \cite{bf, Gastaldi}. For example, for all $\vq\in
(W^{m,p}(\Omega))^2$ where $\frac12< m\le j+1$ and $2\le p\le
\infty$, we have
\begin{eqnarray} \label{eq:bpi-prop}
Q_0(\nabla\cdot\vq) &=& \nabla\cdot \bPi_h \vq, \qquad\mbox{if in addition }\vq\in H(div,\Omega),\\
\|\vq-\bPi_h \vq\|_{L^p(\Omega)}  &\lesssim& h^{m}
\|\vq\|_{W^{m,p}(\Omega)}.\label{eq:bpi-prop-2}
\end{eqnarray}
It is also well-known that for all $0\le m\le j+1$,
\begin{equation} \label{eq:bp-prop}
  \|\vq-\bP_h \vq\| \lesssim h^{m} \|\vq\|_m.
\end{equation}
Using the above estimates and the triangle inequality, one can
easily derive the following estimate
\begin{equation} \label{a2}
\|\bPi_h\nabla v-\bP_h\nabla v\| \lesssim h^{m}\|v\|_{m+1}
\end{equation}
for all $v\in H^{m+1}(\Omega)$ where $\frac{1}{2}< m \le j+1$.

\medskip
Next, we shall present some useful relations for the discrete weak
gradient $\nabla_w$, the projection operator $\bP_h$, and the
interpolation $\bPi_h$. The results can be summarized as follows.

\medskip
\begin{lemma} Let $\gamma>\frac12$ be any real number. The following results hold true.
\begin{itemize}
\item[(i)] For any $v \in H^1(\Omega)$, we have
\begin{equation} \label{commutative}
\nabla_w (Q_h v)  = \bP_h (\nabla v).
\end{equation}
\item[(ii)] For any $\bq\in  (H^\gamma(\Omega))^2\cap H(div,\Omega)$ and
$v_h=\{v_0,v_b\}\in V_{h}$, we have
\begin{equation}\label{div-q}
(\nabla\cdot\bq, \;v_0) = -(\bPi_h\bq, \;\nabla_w v_h) +
\sum_{e\in\E_h\cap\partial\Omega}
\langle(\bPi_h\bq)\cdot\bn,v_b\rangle_e.
\end{equation}
In particular, if either $v_h\in V_{0,h}$ or $\bq \cdot\bn=0$ on
$\partial\Omega$, then
\begin{equation}\label{div-q-homogeneous}
(\nabla\cdot\bq, \;v_0) = - (\bPi_h\bq, \;\nabla_w v_h).
\end{equation}
\end{itemize}
\end{lemma}

\begin{proof}
To prove (\ref{commutative}), we first recall the following
well-known relation \cite{bf}
$$
\nabla\cdot RT_j(K) = P_j(K_0),\qquad RT_j(K)\cdot \bn|_e = P_j(e).
$$
Thus, for any $\bw\in \Sigma_h$ and $K\in\T_h$, by the definition of
$\nabla_w$ and properties of the $L^2$ projection, we have
$$
  \begin{aligned}
    (\nabla_w Q_h v, \bw)_K &= -(Q_0 v, \nabla\cdot \bw)_K+
    \langle Q_bv, \bw\cdot\vn\rangle_{\partial K}\\
    &= -(v, \nabla\cdot \bw)_K + \langle v, \bw\cdot\vn\rangle_{\partial K}\\
    &= (\nabla v, \bw)_K\\
    &= (\bP_h\nabla v, \bw)_K,
  \end{aligned}
$$
which implies (\ref{commutative}). As to (\ref{div-q}), using the
fact that $\nabla\cdot RT_j(K) = P_j(K_0)$, the property
(\ref{eq:bpi-prop}), and the definition of $\nabla_w$ we obtain
\begin{eqnarray*}
(\nabla\cdot\bq, \;v_0)&=& (Q_0(\nabla\cdot\bq), \;v_0)=(\nabla\cdot \bPi_h\bq, \;v_0)\\
&=&-\sum_{K\in {\cal T}_h}(\bPi_h\bq, \nabla_w v_h)_K
   + \sum_{K\in {\cal T}_h} \langle v_b, \bPi_h \bq\cdot \bn\rangle_{\partial K} \\
&=&-\sum_{K\in {\cal T}_h}(\bPi_h \bq, \nabla_w v_h)_K +
\sum_{e\in\T_h\cap\partial\Omega} \langle
(\bPi_h\bq)\cdot\bn,v_b\rangle_e.
\end{eqnarray*}
This completes the proof of (\ref{div-q}). The equality
(\ref{div-q-homogeneous}) is a direct consequence of (\ref{div-q})
since the boundary integrals vanish under the given condition.
\end{proof}

\subsection{Discrete Norms and Inequalities}
Let $v_h=\{v_0, v_b\}\in V_h$. Define on each $K\in \T_h$
$$
\begin{aligned}
\| v_h \|_{0,h,K}^2 &= \|v_0\|_{0,K}^2 +  h \|v_0-v_b\|_{\partial K}^2, \\
\| v_h \|_{1,h,K}^2 &= \|v_0\|_{1,K}^2 +  h^{-1} \|v_0-v_b\|_{\partial K}^2, \\
| v_h |_{1,h,K}^2 &= |v_0|_{1,K}^2 +  h^{-1} \|v_0-v_b\|_{\partial K}^2.
\end{aligned}
$$
Using the above quantities, we define the following discrete norms
and semi-norms for the finite element space $V_h$
$$
\begin{aligned}
\| v_h \|_{0,h} &:= \left(\sum_{K\in \T_h}\| v_h \|_{0,h,K}^2
\right)^{1/2},\\
\| v_h \|_{1,h} &:= \left(\sum_{K\in \T_h}\| v_h \|_{1,h,K}^2 \right)^{1/2}, \\
| v_h |_{1,h} &:= \left(\sum_{K\in \T_h}| v_h |_{1,h,K}^2
\right)^{1/2}.
\end{aligned}
$$
It is clear that $\| v_h \|_{0,h}^2 = \ldp v_h, v_h\rdp$. Hence,
$\|\cdot \|_{0,h}$ provides a discrete $L^2$ norm for $V_h$. It is
not hard to see that $|\cdot|_{1,h}$ and $\|\cdot\|_{1,h}$ define a
discrete $H^1$ semi-norm and a norm for $V_h$, respectively. Observe
that $|v_h|_{1,h}=0$ if and only if $v_h\equiv constant$. Thus,
$|\cdot|_{1,h}$ is a norm in $V_{0,h}$ and $\bar{V}_h$.

For any $K\in \T_h$ and $e$ being an edge of $K$, the following
trace inequality is well-known
\begin{equation}\label{boundary2interior}
\|g\|_{e}^2 \lesssim  h^{-1} \|g\|_K^2 + h^{2s-1} |g|_{s,K}^2,\quad
\frac12 < s \le 1,
\end{equation}
for all $g\in H^1(K)$. Here $|g|_{s,K}$ is the semi-norm in the
Sobolev space $H^s(K)$. The inequality (\ref{boundary2interior}) can
be verified through a scaling argument for the standard Sobolev
trace inequality in $H^s$ with $s\in (\frac12, 1]$. If $g$ is a
polynomial in $K$, then we have from (\ref{boundary2interior}) and
the standard inverse inequality that
\begin{equation}\label{boundary2interior-discrete}
\|g\|_{e}^2 \lesssim  h^{-1} \|g\|_K^2.
\end{equation}

From (\ref{boundary2interior-discrete}) and the triangle inequality,
it is not hard to see that for any $v_h\in V_h$ one has
$$
\left(\sum_{K\in \T_h} (\|v_0\|_{0,K}^2 +  h \|v_b\|_{\partial K}^2) \right)^{1/2} \lesssim \|v_h\|_{0,h} \lesssim
\left(\sum_{K\in \T_h} (\|v_0\|_{0,K}^2 +  h \|v_b\|_{\partial K}^2) \right)^{1/2}.
$$
In the rest of this paper, we shall use the above equivalence
without particular mentioning or referencing.

The following Lemma establishes an equivalence between the two
semi-norms $|\cdot|_{1,h}$ and $\|\nabla_w\cdot\|$.
\medskip
\begin{lemma} \label{lem:discretenorm-equivalence}
  For any $v_h=\{v_0, v_b\}\in V_h$, we have
  \begin{equation} \label{eq:discreteh1semi}
  |v_h|_{1,h} \lesssim \|\nabla_w v_h\| \lesssim | v_h|_{1,h}.
  \end{equation}
\end{lemma}

\begin{proof}
Using the definition of $\nabla_w$, integration by parts, the
Schwarz inequality, the inequality
(\ref{boundary2interior-discrete}), and the Young's inequality, we
have
$$
\begin{aligned}
  \|\nabla_w v_h\|_K^2 &= -(v_0, \nabla\cdot\nabla_w v_h)_K +
  \langle v_b, \nabla_w v_h\cdot\vn\rangle_{\partial K} \\
  &= \langle v_b-v_0, \nabla_w v_h\cdot\vn\rangle_{\partial K} + (\nabla v_0, \nabla_w v_h)_K \\
  & \le \|v_0-v_b\|_{\partial K} \|\nabla_w v_h\cdot\vn\|_{\partial K} + \|\nabla v_0\|_K \|\nabla_w v_h\|_K \\
  & \lesssim \|v_0-v_b\|_{\partial K} h^{-\frac12} \|\nabla_w v_h\|_K + \|\nabla v_0\|_K \|\nabla_w v_h\|_K \\
  &\lesssim \|\nabla_w v_h\|_K\left(\|\nabla v_0\|_K +  h^{-\frac12} \|v_0-v_b\|_{\partial K}\right).
\end{aligned}
$$
This completes the proof of $\|\nabla_w v_h\| \lesssim |
v_h|_{1,h}$.

To prove $ |v_h|_{1,h} \lesssim \|\nabla_w v_h\|$, let $K\in\T_h$ be
any element and consider the following subspace of $RT_j(K)$
$$
D(j,K):=\{\vq\in RT_j(K): \ \vq\cdot \vn =0 \mbox{\  on\ } \partial
K\}.
$$
Note that $D(j,K)$ forms a dual of $(P_{j-1}(K))^2$. Thus, for any
$\nabla v_0\in (P_{j-1}(K))^2$, one has
\begin{equation}\label{jw.08}
\|\nabla v_0\|_K = \sup_{\vq \in D(j,K)} \frac{(\nabla v_0,
\vq)_K}{\|\vq\|_K}.
\end{equation}
It follows from the integration by parts and the definition of
$\nabla_w$ that
$$
(\nabla v_0, \vq)_K = -(v_0,\nabla\cdot\vq)_K=(\nabla_w v_h, \vq)_K,
$$
which, together with (\ref{jw.08}) and the Cauchy-Schwarz
inequality, gives
\begin{equation}\label{jw.09}
 \|\nabla v_0\|_K \le \|\nabla_w v_h\|_K.
\end{equation}
Note that for $j=0$, we have $\nabla v_0 = 0$ and the above
inequality is satisfied trivially.

Analogously, let $e$ be an edge of $K$ and denote by $D_e(j,K)$ the
collection of all $\vq\in RT_j(K)$ such that all degrees of freedom,
except those for $\vq\cdot\vn|_e$, vanish. It is well-known that
$D_e(j,K)$ forms a dual of $P_j(e)$. Thus, we have
\begin{equation}\label{jw.18}
\|v_0-v_b\|_e = \sup_{\vq \in D_e(j,K)} \frac{\langle v_0-v_b,
\vq\cdot\bn\rangle_{e}}{\|\vq\cdot\bn\|_e}.
\end{equation}
It follows from (\ref{discrete-weak-gradient-new}) and the
integration by parts on $(v_0,\nabla\cdot\vq)_K$ that
\begin{equation}\label{jw.28}
(\nabla_w v_h, \vq)_K=(\nabla v_0, \vq)_K+\langle v_b-v_0,
\vq\cdot\bn\rangle_{\partial K},\qquad \forall \ \vq\in RT_j(K).
\end{equation}
In particular, for $\vq\in D_e(j,K)$, we have
$$
(\nabla v_0, \vq)_K =0,\qquad \langle v_b-v_0,
\vq\cdot\bn\rangle_{\partial K}=\langle v_b-v_0,
\vq\cdot\bn\rangle_{e}.
$$
Substituting the above into (\ref{jw.28}) yields
\begin{equation}\label{jw.38}
(\nabla_w v_h, \vq)_K=\langle v_b-v_0, \vq\cdot\bn\rangle_{e},\qquad
\forall \ \vq\in D_e(j,K).
\end{equation}
Using the Cauchy-Schwarz inequality we arrive at
$$
|\langle v_b-v_0, \vq\cdot\bn\rangle_{e}|\leq \|\nabla_w v_h\|_K \
\|\vq\|_K,
$$
for all $\vq\in D_e(j,K)$. By the scaling argument, for such $\vq\in
D_e(j,K)$, we have $\|\vq\|_K\lesssim h^{\frac12}
\|\vq\cdot\vn\|_e$. Thus, we obtain
$$
|\langle v_b-v_0, \vq\cdot\bn\rangle_{e}|\lesssim h^{\frac12}
\|\nabla_w v_h\|_K \ \|\vq\cdot\bn\|_e,\qquad \forall\bq\in
D_e(j,K),
$$
which, together with (\ref{jw.18}), implies the following estimate
$$
\|v_0-v_b\|_e\lesssim h^{\frac12} \|\nabla_w v_h\|_K.
$$
Combining the above estimate with (\ref{jw.09}) gives a proof of $|
v_h|_{1,h} \lesssim \|\nabla_w v_h\|$. This completes the proof of
(\ref{eq:discreteh1semi}).
\end{proof}
\medskip

The discrete semi-norms satisfy the usual inverse inequality, as
stated in the following Lemma.
\medskip
\begin{lemma}
  For any $v_h=\{v_0, v_b\}\in V_h$, we have
  \begin{equation} \label{eq:discreteinverse}
    |v_h|_{1,h} \lesssim h^{-1} \|v_h\|_{0,h}.
  \end{equation}
Consequently, by combining (\ref{eq:discreteh1semi}) and (\ref{eq:discreteinverse}), we have
\begin{equation}
  \|\nabla_w v_h\| \lesssim h^{-1} \| v_h\|_{0,h}.
\end{equation}
\end{lemma}

\begin{proof}
The proof follows from the standard inverse inequality and the
definition of $\|\cdot\|_{0,h}$ and $|\cdot|_{1,h}$; details are
thus omitted.
\end{proof}
\medskip

Next, let us show that the discrete semi-norm $\|\nabla_w
(\cdot)\|$, which is equivalent to $|\cdot|_{1,h}$ as proved in
Lemma \ref{lem:discretenorm-equivalence}, satisfies a
Poincar\'{e}-type inequality.

\medskip
\begin{lemma}
The Poincar\'{e}-type inequality holds true for functions in
$V_{0,h}$ and $\bar{V}_h$. In other words, we have the following
estimates:
\begin{align}
\| v_h \|_{0,h} &\lesssim \|\nabla_w v_h\|\qquad \forall \ v_h\in V_{0,h},\label{eq:discretepoincare1}\\
\| v_h \|_{0,h} &\lesssim \|\nabla_w v_h\|\qquad \forall \ v_h\in
\bar{V}_h.\label{eq:discretepoincare2}
\end{align}
\end{lemma}

\begin{proof}
For any $v_h\in V_{0,h}$, let $\bq\in (H^1(\Omega))^2$ be such that
$\nabla\cdot\bq = v_0$ and $\|\bq\|_1 \lesssim \|v_0\|$. Such a
vector-valued function $\bq$ exists on any polygonal domain
\cite{arnold-regularinversion}. One way to prove the existence of
$\bq$ is as follows. First, one extends $v_h$ by zero to a convex
domain which contains $\Omega$. Secondly, one considers the Poisson
equation on the enlarged domain and set $\vq$ to be the flux. The
required properties of $\vq$ follow immediately from the full
regularity of the Poisson equation on convex domains. By
(\ref{eq:bpi-prop}), we have
$$
\|\bPi_h \bq\|\lesssim \|\bq\|_1 \lesssim \|v_0\|.
$$
Consequently, by (\ref{div-q-homogeneous}) and the Schwarz inequality,
$$
\|v_0\|^2 = (v_0, \nabla\cdot\bq) = -(\bPi_h \bq, \nabla_w v_h)
\lesssim \|v_0\| \|\nabla_w v_h\|.
$$
It follows from Lemma \ref{lem:discretenorm-equivalence} that
$$
\sum_{K\in \T_h} h \|v_0-v_b\|_{\partial K}^2 \lesssim \sum_{K\in
\T_h} h^{-1} \|v_0-v_b\|_{\partial K}^2 \le | v_h|_{1,h}^2 \lesssim
\|\nabla_w v_h\|^2.
$$
Combining the above two estimates gives a proof of the inequality
(\ref{eq:discretepoincare1}).

As to (\ref{eq:discretepoincare2}), since $v_h\in \bar{V}_h$ has
mean value zero, one may find a vector-valued function $\bq$
satisfying $\nabla\cdot\bq= v_0$ and $\bq\cdot\bn = 0$ on
$\partial\Omega$ (see \cite{arnold-regularinversion} for details).
In addition, we have $\|\bq\|_1 \lesssim \|v_0\|$. The rest of the
proof follows the same avenue as the proof of
(\ref{eq:discretepoincare1}).
\end{proof}
\medskip

Next, we shall introduce a discrete norm in the finite element space
$V_{0,h}$ that plays the role of the standard $H^2$ norm. To this
end, for any internal edge $e \in \E_h$, denote by $K_1$ and $K_2$
the two triangles sharing $e$, and by $\bn_1$, $\bn_2$ the outward
normals with respect to $K_1$ and $K_2$. Define the jump on $e$ by
$$
\ljump \nabla_w \psi_h\cdot \bn\rjump =
(\nabla_w\psi_h)|_{K_1}\cdot\bn_1 +
(\nabla_w\psi_h)|_{K_2}\cdot\bn_2.
$$
If the edge $e$ is on the boundary $\partial\Omega$, then there is
only one triangle $K$ which admits $e$ as an edge. The jump is then
modified as
$$
\ljump\nabla_w \psi_h\cdot \bn\rjump = (\nabla_w\psi_h)|_{K} \cdot
\bn.
$$
For $\psi_h\in V_{0,h}$, define
\begin{equation}\label{triple-bar-h2}
\3bar \psi_h \3bar = \left(\sum_{K\in\T_h} \|\nabla\cdot\nabla_w
\psi_h\|_{K}^2 + \sum_{e\in\E_h}h^{-1} \|\ljump\nabla_w \psi_h\cdot
\bn\rjump\|_e^2) \right)^{1/2}.
\end{equation}

\medskip
\begin{lemma} The map $\3bar\cdot\3bar:\ V_{0,h} \to \mathbb{R}$,
as given in (\ref{triple-bar-h2}), defines a norm in the finite
element space $V_{0,h}$. Moreover, one has
\begin{align}
  (\nabla_w v_h,\nabla_w \psi_h) &\lesssim \|v_h\|_{0,h} \3bar\psi_h\3bar &&\forall \
  v_h\in V_h,\,\psi_h\in V_{0,h}, \label{eq:3barnorm-1} \\
  \sup_{v_h\in V_h} \frac{(\nabla_w v_h, \nabla_w \psi_h)}{\|v_h\|_{0,h}} &\gtrsim
  \3bar\psi_h\3bar &&\forall\ \psi_h\in V_{0,h}. \label{eq:3barnorm-2}
\end{align}
\end{lemma}

\begin{proof} To verify that $\3bar\cdot\3bar$ defines a norm, it is
sufficient to show that $\3bar\psi_h\3bar=0$ implies $\psi_h \equiv
0$. To this end, let $\3bar\psi_h\3bar=0$. It follows that
$\nabla\cdot \nabla_w \psi_h=0$ on each element and $\ljump\nabla_w
\psi_h\cdot \bn\rjump=0$ on each edge. The definition of the
discrete weak gradient $\nabla_w$ then implies the following
$$
(\nabla_w \psi_h, \nabla_w \psi_h) =  \sum_{K\in \T_h}
\left(-(\psi_0, \nabla\cdot\nabla_w \psi_h)_K + \langle\psi_b,
\nabla_w\psi_h \cdot\vn\rangle_{\partial K} \right)=0.
$$
Thus, we have $\nabla_w\psi_h=0$. Since $\psi_h\in V_{0,h}$, then
$\nabla_w \psi_h = 0$ implies $\psi_h\equiv 0$. This shows that
$\3bar\cdot\3bar$ defines a norm in $V_{0,h}$. The inequality
(\ref{eq:3barnorm-1}) follows immediately from the following
identity
$$
(\nabla_w v_h, \nabla_w \psi_h) =  \sum_{K\in \T_h} \left(-(v_0,
\nabla\cdot \nabla_w \psi_h)_K + \langle v_b, \nabla_w\psi_h
\cdot\vn\rangle_{\partial K} \right)
$$
and the Schwarz inequality.

To verify (\ref{eq:3barnorm-2}), we chose a particular $v^*_h \in
V_h$ such that
$$
\begin{aligned}
v^*_0 &= - \nabla\cdot \nabla_w \psi_h \qquad &&\textrm{in } K_0, \\
v^*_b &= h^{-1} \ljump\nabla_w\psi_h\cdot\vn\rjump  \qquad
&&\textrm{on edge } e.
\end{aligned}
$$
It is not hard to see that $\|v^*_h\|_{0,h} \lesssim
\3bar\psi_h\3bar$. Thus, we have
$$
\begin{aligned}
  \sup_{v_h\in V_h} \frac{(\nabla_w v_h, \nabla_w \psi_h)}{\|v_h\|_{0,h}} &\ge \frac{(\nabla_w v^*_h, \nabla_w \psi_h)}{\|v^*_h\|_{0,h}} \\
  &= \frac{\sum_{K\in \T_h} \left(-(v^*_0, \nabla\cdot \nabla_w \psi_h)_K + \langle v^*_b,
  \nabla_w\psi_h \cdot\vn\rangle_{\partial K} \right)}{\|v^*_h\|_{0,h}} \\
  &= \frac{\3bar\psi_h\3bar^2}{\|v^*_h\|_{0,h}}\gtrsim \3bar\psi_h\3bar.
\end{aligned}
$$
This completes the proof of the lemma.
\end{proof}
\medskip

\begin{remark}
Using the boundedness (\ref{eq:3barnorm-1}) and the discrete
Poincare inequality (\ref{eq:discretepoincare1}) we have the
following estimate for all $\psi_h\in V_{0,h}$
$$
\|\nabla_w \psi_h\|^2 = (\nabla_w \psi_h, \nabla_w \psi_h) \lesssim
\|\psi_h\|_{0,h} \3bar\psi_h\3bar \lesssim \|\nabla_w \psi_h\|
\3bar\psi_h\3bar.
$$
This implies that $\|\nabla_w \psi_h\|\lesssim \3bar \psi_h\3bar$.
In other words, $\3bar\cdot\3bar$ is a norm that is stronger than
$\|\cdot\|_{1,h}$. In fact, the norm $\3bar\cdot\3bar$ can be viewed
as a discrete equivalence of the standard $H^2$ norm for smooth
functions with proper boundary conditions.
\end{remark}

Next, we shall establish an estimate for the $L^2$ projection
operator $Q_h$ in the discrete norm $\|\cdot\|_{0,h}$.
\medskip
\begin{lemma} \label{lem:Qh}
 Let $Q_h$ be the $L^2$ projection operator into the finite element space
 $V_h$. Then, for any $v\in H^{m}(\Omega)$ with $\frac12 < m \le j+1$, we have
  \begin{equation} \label{eq:a5}
    \|v-Q_h v\|_{0,h} \lesssim h^m \|v\|_{m}.
  \end{equation}
\end{lemma}

\begin{proof}
For the $L^2$ projection on each element $K$, it is known that the following estimate holds
true
\begin{equation}\label{five-guys.00}
 \|v-Q_0v\|_K \lesssim h^m \|v\|_{m,K}.
\end{equation}
Thus, it suffices to deal with the terms associated with the
edges/faces given by
\begin{equation}\label{five-guys.01}
\sum_K h\|(v-Q_0v)-(v-Q_bv)\|^2_{\partial K}= \sum_K
h\|Q_0v-Q_bv\|^2_{\partial K}.
\end{equation}
Since $Q_b$ is the $L^2$ projection on edges, then we have
$$
\|Q_0v-Q_bv\|^2_{\partial K} \leq \|v-Q_0v\|^2_{\partial K}.
$$
Let $s\in (\frac12, 1]$ be any real number satisfying $s\le m$. It
follows from the above inequality and the trace inequality
(\ref{boundary2interior}) that
$$
\|Q_0v-Q_bv\|^2_{\partial K} \lesssim h^{-1} \|v-Q_0v\|^2_{K} +
h^{2s-1}|v-Q_0v|^2_{s,K}.
$$
Substituting the above into (\ref{five-guys.01}) yields
\begin{eqnarray*}
\sum_K h\|(v-Q_0v)-(v-Q_bv)\|^2_{\partial K}&\lesssim &\sum_K \left(
\|v-Q_0v\|^2_{K} + h^{2s}|v-Q_0v|^2_{s,K}\right)\\
&\lesssim & h^{2m}\|v\|_{m}^2,
\end{eqnarray*}
which, together with (\ref{five-guys.00}), completes the proof of
the lemma.
\end{proof}
\medskip

\subsection{Ritz and Neumann Projections}

To establish an error analysis in the forthcoming section, we shall
introduce and analyze two additional projection operators, the Ritz
projection $R_h$ and the Neumann projection $N_h$, by applying the
weak Galerkin method to the Poisson equation with various boundary
conditions.

For any $v\in H_0^1(\Omega)\cap H^{1+\gamma}(\Omega)$ with
$\gamma>\frac12$, the Ritz projection $R_h v\in V_{0,h}$ is defined
as the unique solution of the following problem:
\begin{equation} \label{eq:Ritz}
(\nabla_w (R_h v),\nabla_w \psi_h) = (\bPi_h \nabla v,  \nabla_w
\psi_h),\qquad \forall\ \psi_h\in V_{0,h}.
\end{equation}
Here $\gamma>\frac12$ in the definition of $R_h$ is imposed to ensure that
$\bPi_h \nabla v$ is well-defined.
From the identity (\ref{div-q-homogeneous}), clearly if $\Delta v\in L^2(\Omega)$,
then $R_h v$ is identical to the weak Galerkin finite element solution \cite{WangYe_PrepSINUM_2011} to the Poisson
equation with homogeneous Dirichlet boundary condition for which $v$ is the
exact solution.
Analogously, for any $v\in \bar{H}^1(\Omega)\cap H^{1+\gamma}(\Omega)$ with $\gamma>\frac12$, we
define the Neumann projection $N_h v\in \bar{V}_h$ as the solution
to the following problem
\begin{equation} \label{eq:Neumann}
(\nabla_w (N_h v),\nabla_w \psi_h)  = (\bPi_h \nabla v,  \nabla_w
\psi_h),\qquad \forall \ \psi_h\in \bar{V}_h.
\end{equation}
It is useful to note that the above equation holds true for all
$\psi_h\in V_h$ as $\nabla_w 1 = 0$. Similarly, if $\Delta v\in
L^2(\Omega)$ and in addition $\partial v/\partial \vn =0$ on
$\partial \Omega$, then $N_h v$ is identical to the weak Galerkin
finite element solution to the Poisson equation with homogeneous
Neumann boundary condition, for which $v$ is the exact solution. The
well-posedness of $R_h$ and $N_h$ follows immediately from the
Poincar\'{e}-type inequalities (\ref{eq:discretepoincare1}) and
(\ref{eq:discretepoincare2}).

Using (\ref{commutative}), it is easy to see that for all $\psi_h\in
V_{0,h}$ we have
\begin{equation} \label{eq:Ritz-orthogonal}
  (\nabla_w (Q_hv-R_h v),\nabla_w \psi_h) = ((\bP_h - \bPi_h) \nabla v,  \nabla_w \psi_h).
\end{equation}
And similarly, for all $\psi_h\in \bar{V}_h$,
\begin{equation} \label{eq:Neumann-orthogonal}
  (\nabla_w (Q_hv-N_h v),\nabla_w \psi_h) = ((\bP_h - \bPi_h) \nabla v,  \nabla_w \psi_h).
\end{equation}
From the definitions of $\bar{V}_h$ and $Q_h$, clearly $Q_h$ maps
$\bar{H}^1(\Omega)$ into $\bar{V}_h$.

For convenience, let us adopt the following notation
\begin{eqnarray*}
\{R_0 v, R_b v\}:=R_h v, \qquad \{N_0 v, N_b v\}:=N_h v,
\end{eqnarray*}
where again the subscript ``$0$'' denotes the function value in the
interior of triangles, while ``$b$'' denotes the trace on $\E_h$.
For Ritz and Neumann projections, the following approximation error
estimates hold true.

\medskip
\begin{lemma} \label{lem:elliptic-error}
For $v\in H_0^1(\Omega) \cap H^{m+1}(\Omega)$ or $\bar{H}^1(\Omega) \cap H^{m+1}(\Omega)$, where $\frac{1}{2}< m\le j+1$, we have
\begin{align}
  \|\nabla_w(Q_h v-R_h v)\| &\lesssim h^{m} \|v\|_{m+1}, \label{eq:Ritz-error-H1}\\
  \|\nabla_w(Q_h v-N_h v)\| &\lesssim h^{m} \|v\|_{m+1}. \label{eq:Neumann-error-H1}
\end{align}
Moreover, assume $\Delta v\in L^2(\Omega)$ and
that the Poisson problem in $\Omega$ with either the homogeneous Dirichlet boundary condition or
the homogeneous Neumann boundary condition has $H^{1+s}$ regularity, where $\frac{1}{2}< s \le 1$, then
\begin{align}
  \| Q_0 v-R_0 v\| &\lesssim h^{m+s} \|v\|_{m+1} + h^{1+s} \|(I-Q_0)\Delta v\|, \label{eq:Ritz-error-L2} \\
  \| Q_0 v-N_0 v\| & \lesssim h^{m+\min(s,j+\frac12)} \|v\|_{m+1} + h^{1+s} \|(I-Q_0)\Delta v\|. \label{eq:Neumann-error-L2}
\end{align}
\end{lemma}

\begin{proof} The estimates (\ref{eq:Ritz-error-H1})-(\ref{eq:Neumann-error-H1}) follow
immediately from
(\ref{eq:Ritz-orthogonal})-(\ref{eq:Neumann-orthogonal}),
(\ref{a2}), and the Schwarz inequality. Next, we prove
(\ref{eq:Neumann-error-L2}) by using the standard duality argument.
Let $\phi\in \bar{H}^1(\Omega)$ be the solution of $-\Delta \phi =
Q_0 v-N_0 v$ with boundary condition $\left.\frac{\partial
\phi}{\partial\vn}\right|_{\partial \Omega} = 0$. Note that $\phi$
is well-defined since $Q_h v-N_h v \in \bar{V}_h$. According to the
regularity assumption, we have $\phi\in H^{1+s}(\Omega)$ and
$\|\phi\|_{{1+s}}\lesssim \|Q_0 v-N_0 v\|$. Then, by
(\ref{div-q-homogeneous}), (\ref{eq:Neumann-orthogonal}), the
Schwarz inequality and (\ref{a2}), we arrive at
$$
\begin{aligned}
  \| Q_0 v-N_0 v\|^2 &= (Q_0 v-N_0 v, -\Delta \phi) = (\bPi_h \nabla \phi, \nabla_w (Q_h v-N_h v)) \\
  &= (\bPi_h \nabla\phi - \nabla_w(N_h\phi), \nabla_w (Q_h v-N_h v)) + ((\bP_h-\bPi_h)\nabla v, \nabla_w (N_h\phi)) \\
  &\le \bigg(\| \bPi_h \nabla\phi - \bP_h\nabla\phi\|+\|\nabla_w (Q_h\phi-N_h\phi)\|\bigg) \| \nabla_w (Q_h v-N_h v)\|  \\
  &\qquad    + ((\bP_h-\bPi_h)\nabla v, \nabla_w (N_h\phi-Q_h\phi))+ ((\bP_h-\bPi_h)\nabla v, \bP_h\nabla \phi) \\
  & \lesssim h^{m+s}\|\phi\|_{1+s} \|v\|_{m+1} + ((I-\bPi_h)\nabla v, \bP_h\nabla \phi).
\end{aligned}
$$
Using integration by parts, the triangular inequality and the definition of $\bPi_h$, we have
\begin{equation} \label{eq:duality-component}
\begin{aligned}
  &((I-\bPi_h)\nabla v, \bP_h\nabla \phi) \\
  =& ((I-\bPi_h)\nabla v, (\bP_h-I)\nabla \phi) + ((I-\bPi_h)\nabla v, \nabla\phi) \\
  \lesssim& h^{m+s}\|\phi\|_{1+s} \|v\|_{m+1} + ((I-\bPi_h)\nabla v\cdot\vn,\phi)_{\partial\Omega} - (\nabla\cdot(I-\bPi_h)\nabla v, \phi) \\
  =& h^{m+s}\|\phi\|_{1+s} \|v\|_{m+1} + ((I-\bPi_h)\nabla v\cdot\vn,\phi-Q_b\phi)_{\partial\Omega} - ((I-Q_0) \Delta v, \phi) \\
  \lesssim& h^{m+s}\|\phi\|_{1+s} \|v\|_{m+1} + (h^{m-\frac12} \|v\|_{m+\frac12,\partial\Omega})
  (h^{\min(s+\frac12, j+1)}\|\phi\|_{s+\frac12,\partial\Omega}) \\
  &\qquad     - ((I-Q_0) \Delta v, (I-Q_0) \phi) \\
  \lesssim & h^{m+\min(s,j+\frac12)}\|\phi\|_{1+s} \|v\|_{m+1} + h^{1+s} \|\phi\|_{1+s}\|(I-Q_0)\Delta v\|.
\end{aligned}
\end{equation}
In the proof of (\ref{eq:duality-component}), we have used the fact
that $\Pi_h(\nabla v\cdot\vn)$ is exactly the $L^2$ projection of
$\nabla v\cdot\vn$ on $\partial\Omega$. Combining the above gives
$$
\begin{aligned}
\| Q_0 v-N_0 v\|^2 &\lesssim \bigg(h^{m+\min(s,j+\frac12)}\|v\|_{m+1}+h^{1+s}\|(I-Q_0)\Delta v\|\bigg)\|\phi\|_{1+s} \\
&\lesssim
\bigg(h^{m+\min(s,j+\frac12)}\|v\|_{m+1}+h^{1+s}\|(I-Q_0)\Delta
v\|\bigg)\| Q_0 v-N_0 v\|.
\end{aligned}
$$
This completes the proof of the estimate
(\ref{eq:Neumann-error-L2}). The inequality (\ref{eq:Ritz-error-L2})
can be verified in a similar way by considering a function $\phi\in
H_0^1(\Omega)$ satisfying a Poisson equation with homogeneous
Dirichlet boundary condition. Observe that in this case, the
boundary integral $((I-\bPi_h)\nabla
v\cdot\vn,\phi)_{\partial\Omega}$ in inequality
(\ref{eq:duality-component}) shall vanish due to the vanishing value
of $\phi$.
\end{proof}
\medskip

\begin{remark}
It is not hard to see from (\ref{eq:duality-component}) that for the
Neumann projection, if in addition we have $\frac{\partial
v}{\partial n}=0$ on $\partial\Omega$, then the term
$((I-\bPi_h)\nabla v\cdot\vn,\phi)_{\partial\Omega}$ vanishes and
one obtains the optimal order estimate of $h^{m+s}$ instead of
$h^{m+\min(s,j+\frac12)}$ for the Neumann projection operator.
\end{remark}

\begin{remark} \label{rem:fullregularity-elliptic}
If the Poisson equation has the full $H^2$ regularity in $\Omega$,
then for $v$ satisfying the assumptions of Lemma
\ref{lem:elliptic-error}, we have
$$
\begin{aligned}
  \| Q_0 v-R_0 v\| &\lesssim h^{m+1} \|v\|_{m+1} + h^2\|(I-Q_0)\Delta v\|\qquad \textrm{for }\frac{1}{2}< m\le j+1,\\[2mm]
  \| Q_0 v-N_0 v\| &\lesssim \begin{cases} h^{m+\frac12} \|v\|_{m+1}+ h^2\|(I-Q_0)\Delta v\|
    \quad &\textrm{for }j=0,\, \frac{1}{2}< m\le 1, \\ h^{m+1} \|v\|_{m+1}+ h^2\|(I-Q_0)\Delta v\| \quad &\textrm{for }j\ge 1,\, \frac{1}{2}< m\le j+1. \end{cases}
\end{aligned}
$$
Again, if in addition, $\frac{\partial v}{\partial n}=0$ on
$\partial\Omega$, then the Neumann projection has optimal order of
error estimates, even for $j=0$.
\end{remark}

\begin{remark} \label{rem:ellipticL2error}
The duality argument used in Lemma \ref{lem:elliptic-error} works
only for $\|Q_0 v- R_0 v\|$ and $\|Q_0 v- N_0 v\|$. For $\|Q_h v-
R_h v\|_{0,h}$ and $\|Q_h v- N_h v\|_{0,h}$ involving element
boundary information, we currently have only sub-optimal estimates.
More precisely, for $v$ satisfying the assumptions in Lemma
\ref{lem:elliptic-error}, the following estimates hold true.
  \begin{equation} \label{eq:elliptic-discrete-estimate}
    \begin{aligned}
    \| Q_h v-R_h v\|_{0,h} & \lesssim \|\nabla_w(Q_h v-R_h v)\| \lesssim h^{m} \|v\|_{m+1} \quad&&\textrm{for } \frac{1}{2}< m\le j+1,\\
    \| Q_h v-N_h v\|_{0,h} & \lesssim \|\nabla_w(Q_h v-N_h v)\| \lesssim h^{m} \|v\|_{m+1} \quad&&\textrm{for } \frac{1}{2}< m\le j+1.\\
    \end{aligned}
  \end{equation}
Although numerical experiments in \cite{MuWangWangYe} suggest an
optimal order of convergence in the $\|\cdot\|_{0,h}$ norm, it
remains to see if optimal order error estimates hold true or not
theoretically.
\end{remark}
\medskip

Another important observation is that, for sufficiently smooth $v$,
$\nabla_w R_h v$ is identical to the mixed finite element
approximation of $\nabla v$, discretized by using $RT_j$ and
discrete $P_j$ elements. Indeed, we have the following lemma:
\medskip
\begin{lemma}
For any $v\in H_0^1\cap H^{1+\gamma}(\Omega)$ with $\gamma>\frac12$
and $\Delta v\in L^2(\Omega)$, let $\vq_h\in \Sigma_h\cap
H(div,\Omega)$ and $v_0\in L^2(\Omega)$ be piecewise $P_j$
polynomials solving
\begin{equation} \label{eq:mixed-elliptic}
  \begin{cases}
  (\vq_h, \vchi_h) - (\nabla\cdot \vchi_h, v_0) = 0\quad & \forall\, \vchi_h \in \Sigma_h\cap H(div,\Omega), \\
  (\nabla\cdot \vq_h, \psi_0) = (\Delta v, \psi_0)\quad & \forall\, \psi_0 \in  L^2(\Omega)
  \textrm{ piecewise }P_j\textrm{ polynomials}.
  \end{cases}
\end{equation}
In other words, $\vq_h$ and $v_0$ are the mixed finite element
solution, discretized using the $RT_j$ element, to the Poisson
equation with homogeneous Dirichlet boundary condition for which $v$
is the exact solution. Then, one has $\nabla_w R_h v = \vq_h$.
\end{lemma}
\medskip
\begin{proof}
We first show that $\nabla_w R_h v \in \Sigma_h\cap H(div,\Omega)$
by verifying that $(\nabla_w R_h v)\cdot \vn$ is continuous across
internal edges. Let $e\in \E_h\backslash\partial\Omega$ be an
internal edge and $K_1$, $K_2$ be two triangles sharing $e$. Denote
$\bn_1$ and $\bn_2$ the outward normal vectors on $e$, with respect
to $K_1$ and $K_2$, respectively. Let $\psi_h \in V_{0,h}$ satisfy
$\psi_b|_e \neq 0$ and $\psi_0$, $\psi_b$ vanish elsewhere. By the
definition of $R_h$, $\nabla_w$ and the fact that $\bPi_h\nabla v\in
H(div,\Omega)$, we have
$$
\begin{aligned}
0 &= (\bPi_h\nabla v - \nabla_w R_h v, \nabla_w \psi_h) \\
&= (\bPi_h\nabla v - \nabla_w R_h v, \nabla_w \psi_h)_{K_1} + (\bPi_h\nabla v - \nabla_w R_h v, \nabla_w \psi_h)_{K_2} \\
&= ((\bPi_h\nabla v - \nabla_w R_h v)|_{K_1}\cdot \vn_1 + (\bPi_h\nabla v - \nabla_w R_h v)|_{K_2}\cdot \vn_2, \psi_b)_e \\
&= -(\nabla_w R_h v|_{K_1}\cdot \vn_1 + \nabla_w R_h v|_{K_2}\cdot
\vn_2, \psi_b)_e.
\end{aligned}
$$
The above equation holds true for all $\psi_b|_e\in P_j(e)$. Since
$\nabla_w R_h v|_{K_1}\cdot \vn_1 + \nabla_w R_h v|_{K_2}\cdot
\vn_2$ is also in $P_j(e)$, therefore it must be $0$. This completes
the proof of  $\nabla_w R_h v \in H(div,\Omega)$.

Next, we prove that $\nabla_w R_h v$ is identical to the solution
$\vq_h$ of (\ref{eq:mixed-elliptic}). Since the solution to
(\ref{eq:mixed-elliptic}) is unique, we only need to show that
$\nabla_w R_h v$, together with a certain $v_0$, satisfies both
equations in (\ref{eq:mixed-elliptic}). Consider the test function
$\psi_h\in V_{0,h}$ with the form $\psi_h = \{\psi_0,0\}$. By the
definition of $\nabla_w$, equations (\ref{eq:Ritz}) and
(\ref{div-q-homogeneous}), we have
$$
(\nabla\cdot\nabla_w R_h v, \psi_0) = - (\nabla_w R_h v, \nabla_w
\psi_h) = - (\bPi_h\nabla v, \nabla_w \psi_h) = (\Delta v, \psi_0).
$$
Hence $\nabla_w R_h v$ satisfies the second equation of
(\ref{eq:mixed-elliptic}). Now, note that $\nabla\cdot$ is an onto
operator from $\Sigma_h\cap H(div,\Omega)$ to the space of piecewise
$P_j$ polynomials, which allows us to define a $v_0$ that satisfies
the first equation in (\ref{eq:mixed-elliptic}) with $\vq_h$ set to
be $\nabla_w R_h v$. This completes the proof the the lemma.
\end{proof}
\medskip

\begin{remark}
Using the same argument and noticing that (\ref{eq:Neumann}) holds
for all $\psi_h\in V_h$, one can analogously prove that for $v\in
\bar{H}^1(\Omega)\cap H^{1+\gamma}(\Omega)$ with $\gamma>\frac12$
and $\Delta v\in L^2(\Omega)$,
$$
 \nabla_w N_h v \in \Sigma_h \cap H(div, \Omega),
$$
and
$$
\nabla\cdot \nabla_w N_h v = Q_0\Delta v.
$$
\end{remark}

Because $\nabla_w R_h v$ is identical to the mixed finite element
solution to the Poisson equation, by \cite{wang1, Gastaldi}, we have
the following quasi-optimal order $L^{\infty}$ estimate:
\begin{equation} \label{eq:maximumnormestimate1}
\|\nabla v - \nabla_w R_h v\|_{L^{\infty}(\Omega)} \lesssim
h^{n+1}|\ln h| \|\Delta v\|_{W^{n,\infty}(\Omega)},
\end{equation}
for $0\le n\le j$. Furthermore, for $j\ge 1$ and $v\in
W^{j+2,\infty}(\Omega)$, we have the following optimal order error
estimate
\begin{equation} \label{eq:maximumnormestimate2}
  \|\nabla v - \nabla_w R_h v\|_{L^{\infty}(\Omega)} \lesssim h^{n+1} \| v\|_{W^{n+2,\infty}(\Omega)},
\end{equation}
for $1\le n \le j$.

\medskip
Inspired by \cite{Scholz78}, using the above $L^{\infty}$ estimates
we obtain the following lemma, which will play an essential role in
the error analysis to be given in the next section.

\begin{lemma} \label{lem:Linf}
The following quasi-optimal and optimal order error estimates hold
true:
\begin{itemize}
\item[(i)] Let $0\le n \le j$ and $v\in H_0^1(\Omega)\cap W^{n+2,\infty}(\Omega)$. Then for all
$\phi_h = \{v_0, v_b\} \in V_h$, we have
\begin{equation} \label{eq:Linf1}
|(\bPi_h\nabla v - \nabla_w R_h v, \nabla_w \phi_h)| \lesssim
h^{n+\frac12} |\ln h| \| v\|_{W^{n+2,\infty}(\Omega)}
\|\phi_h\|_{0,h}.
\end{equation}
\item[(ii)] Let $j\ge 1$, $1\le n\le j$, and $v\in H_0^1(\Omega)\cap W^{n+2,\infty}(\Omega)$. Then,
for all $\phi_h = \{v_0, v_b\} \in V_h$ we have
\begin{equation} \label{eq:Linf2}
|(\bPi_h\nabla v - \nabla_w R_h v, \nabla_w \phi_h)| \lesssim
h^{n+\frac12} \| v\|_{W^{n+2,\infty}(\Omega)} \|\phi_h\|_{0,h}.
\end{equation}
\end{itemize}
\end{lemma}

\begin{proof}
We first prove part $(i)$. Denote by $\E_{\partial\Omega}$ the set
of all edges in $\E_h\cap \partial\Omega$. For any $e\in
\E_{\partial\Omega}$, let $K_e$ be the only triangle in $\T_h$ that
has $e$ as an edge. Denote by $\T_{\partial\Omega}$ the set of all
$K_e$, for $e\in \E_{\partial\Omega}$. For simplicity of notation,
denote $\vq_h = \bPi_h\nabla v - \nabla_w R_h v$. Since
$(\bPi_h\nabla v - \nabla_w R_h v, \nabla_w \psi_h) = 0$ for all
$\psi_h \in V_{0,h}$, without loss of generality, we only need to
consider $\phi_h$ that vanishes on the interior of all triangles and
all internal edges. Then by the definition of $\phi_h$ and
$\nabla_w$, the scaling argument, and the Schwarz inequality,
$$
\begin{aligned}
  |(\bPi_h\nabla v - \nabla_w R_h v, \nabla_w \phi_h)| &= \left|
  \sum_{K_e\in \T_{\partial\Omega}} (\vq_h, \nabla_w (\phi_b|_e))_{K_e}\right| \\
  &= \left|\sum_{e\in \E_{\partial\Omega}} (\phi_b, \vq_h\cdot\vn)_e \right| \\
  & \lesssim \sum_{e\in \E_{\partial\Omega}} h \|\phi_b\|_{L^{\infty}(e)} \|\vq_h\|_{L^{\infty}(e)} \\
  & \lesssim \|\vq_h\|_{L^{\infty}(\Omega)}  \sum_{e\in \E_{\partial\Omega}} h \left(\|\phi_0\|_{L^{\infty}(K_e)} + \|\phi_0-\phi_b\|_{L^{\infty}(e)} \right) \\
  & \lesssim \|\vq_h\|_{L^{\infty}(\Omega)} \sum_{K_e\in \T_{\partial\Omega}} \|\phi_h\|_{0,h,K_e} \\
  & \lesssim \|\vq_h\|_{L^{\infty}(\Omega)} \left(\sum_{K_e\in \T_{\partial\Omega}} \|\phi_h\|^2_{0,h,K_e}\right)^{\frac12} \left(\sum_{K_e\in \T_{\partial\Omega}} 1 \right)^{\frac12}  \\
  & \lesssim h^{-\frac12} \|\vq_h\|_{L^{\infty}(\Omega)} \|\phi_h\|_{0,h}.
\end{aligned}
$$
Now, by inequalities (\ref{eq:bpi-prop-2}) and (\ref{eq:maximumnormestimate1}), we have
$$
\begin{aligned}
  \|\vq_h\|_{L^{\infty}(\Omega)} &\le \|\nabla v - \bPi_h\nabla v\|_{L^{\infty}(\Omega)} + \|\nabla v - \nabla_w R_h v\|_{L^{\infty}(\Omega)} \\
  &\lesssim h^{n+1}\| v\|_{W^{n+2,\infty}(\Omega)} +  h^{n+1}|\ln h| \|\Delta v\|_{W^{n,\infty}(\Omega)},
\end{aligned}
$$
for $0\le n\le j$. This completes the proof of part $(i)$.

The proof for part $(ii)$ is similar. One simply needs to replace
inequality (\ref{eq:maximumnormestimate1}) by
(\ref{eq:maximumnormestimate2}) in the estimation of
$\|\vq_h\|_{L^{\infty}(\Omega)}$.
\end{proof}

\section{Error analysis} \label{sec:erroranalysis}
The main purpose of this section is to analyze the approximation
error of the weak Galerkin formulation (\ref{eq:wg}). For
simplicity, in this section, we assume that the solution of
(\ref{eq:wg}) satisfies $u\in H^{3+\gamma}(\Omega)$ and $w\in
H^{1+\gamma}(\Omega)$, where $\gamma>\frac{1}{2}$. This is not an
unreasonable assumption, as we know from (\ref{eq:regularity}), the
solution $u$ can have up to $H^4$ regularity as long as $\Omega$
satisfies certain conditions. However, our assumption does not
include all the possible cases for the biharmonic equation.

Testing $w = -\Delta u$ with $\phi_h=\{\phi_0,\phi_b\}\in V_h$ and
then by using (\ref{div-q-homogeneous}) we have
\begin{equation} \label{eq:test-1}
\ldp w, \phi_h\rdp = (w,\phi_0) = - (\nabla\cdot\nabla u, \phi_0) =
(\bPi_h\nabla u,\nabla_w \phi_h).
\end{equation}
Similarly, testing $-\Delta w = f$ with $\psi_h=\{\psi_0,\psi_b\}\in V_{0,h}$ gives
\begin{equation} \label{eq:test-2}
(\bPi_h\nabla w,\nabla_w\psi_h) = (f, \psi_0).
\end{equation}
Comparing (\ref{eq:test-1})-(\ref{eq:test-2}) with the weak Galerkin
form (\ref{eq:wg}), one immediately sees that there is a consistency
error between them. Indeed, since $V_h$ and $V_{0,h}$ are not
subspaces of $H^1(\Omega)$ and $H_0^1(\Omega)$, respectively, the
weak Galerkin method is non-conforming. Therefore, we would like to
first rewrite (\ref{eq:test-1})-(\ref{eq:test-2}) into a form that
is more compatible with (\ref{eq:wg}). By using (\ref{eq:Ritz}) and
(\ref{eq:Neumann}), equations (\ref{eq:test-1})-(\ref{eq:test-2})
can be rewritten as
\begin{equation} \label{eq:mixed}
  \begin{cases}
    \ldp N_h w, \phi_h\rdp - (\nabla_w R_h u,\nabla_h \phi_h) = E(w,u,\phi_h), \\
    (\nabla_w N_h w, \nabla_w\psi_h) = (f,\psi_0),
  \end{cases}
\end{equation}
where
$$
  E(w,u,\phi_h) = \ldp N_h w - w, \phi_h\rdp + (\bPi_h\nabla u -\nabla_w R_h u, \nabla_w \phi_h).
$$

Define $\varepsilon_u = R_h u - u_h\in V_{0,h}$ and $\varepsilon_w = N_h w - w_h \in V_h$. By subtracting (\ref{eq:mixed})
from (\ref{eq:wg}), we have
\begin{equation} \label{eq:error}
  \begin{cases}
    \ldp \varepsilon_w, \phi_h\rdp - (\nabla_w \varepsilon_u,\nabla_h \phi_h) = E(w,u,\phi_h)\qquad &\textrm{for all }\phi_h\in V_h, \\
    (\nabla_w \varepsilon_w, \nabla_w\psi_h) = 0\qquad &\textrm{for all }\psi_h\in V_{0,h}.
  \end{cases}
\end{equation}
Notice here $(\nabla_w \varepsilon_w, \nabla_w\psi_h) = 0$ does not
necessarily imply $\varepsilon_w = 0$, since the equation only holds
for all $\psi_h\in V_{0,h}$ while $\varepsilon_w$ is in $V_h$.

\medskip
\begin{lemma} \label{lem:E1E2}
The consistency error $E(w,u,\phi_h)$ is small in the sense that
$$
|E(w,u,\phi_h)| \lesssim h^m \|w\|_{m+1} \| \phi_h\|_{0,h} + h^{n+\frac12} |\ln h| \| u\|_{W^{n+2,\infty}(\Omega)} \|\phi_h\|_{0,h},
$$
where $\frac{1}{2}<m\le j+1$ and $0\le n\le j$.
Moreover, for $j\ge 1$, we have the improved estimate
$$
    |E(w,u,\phi_h)| \lesssim h^{m} \|w\|_{m+1} \| \phi_h\|_{0,h} + h^{n+\frac12}  \| u\|_{W^{n+2,\infty}(\Omega)} \|\phi_h\|_{0,h},
$$
where $\frac12< m \le j+1$ and $1\le n\le j$.
\end{lemma}

\begin{proof}
The proof is straight forward by using the Schwarz inequality, Lemma \ref{lem:Qh}, Remark \ref{rem:ellipticL2error},
and Lemma \ref{lem:Linf}.
\end{proof}
\medskip

To derive an error estimate from (\ref{eq:error}), let us recall the
standard theory for mixed finite element methods. Given two bounded
bilinear forms $a(\cdot, \cdot)$ defined on $X\times X$ and
$b(\cdot, \cdot)$ defined on $X\times M$, where $X$ and $M$ are
finite dimensional spaces. Denote $X_0\subset X$ by
$$
X_0 = \{ \phi\in X:\: b(\phi, \psi)= 0\textrm{ for all }\psi\in M\}.
$$
Then for all $\chi\in X$ and $\xi\in M$,
$$
\sup_{\phi\in X,\, \psi\in M} \frac{a(\chi,\phi) + b(\phi, \xi) + b(\chi,\psi)}{\|\phi\|_X + \|\psi\|_M} \gtrsim \|\chi\|_X + \|\xi\|_M,
$$
if and only if
\begin{equation} \label{eq:inf-sup}
\begin{aligned}
  \sup_{\phi\in X_0} \frac{a(\chi,\phi)}{\|\phi\|_X} &\gtrsim \|\chi\|_X, \qquad &&\textrm{for all } \chi\in X_0, \\
  \sup_{\phi\in X} \frac{b(\phi, \xi)}{\|\phi\|_X} &\gtrsim \|\xi\|_M, \qquad &&\textrm{for all } \xi\in M.
\end{aligned}
\end{equation}

In our formulation, we set $X = V_h$ with norm $\|\cdot\|_{0,h}$ and
$M=V_{0,h}$ with norm $\3bar\cdot\3bar$. Define
$$
a(\chi,\phi) = \ldp \chi,\phi\rdp,\qquad b(\phi,\xi) = -(\nabla_w
\phi, \nabla_w\xi).
$$
It is not hard to check that both of these bilinear forms are bounded under the given norms.
In particular, the boundedness of $b(\cdot, \cdot)$ has been given in (\ref{eq:3barnorm-1}).
It is also clear that the first inequality in (\ref{eq:inf-sup}) follows from the definition of $a(\cdot,\cdot)$ and $\|\cdot\|_{0,h}$,
and the second inequality follows directly from (\ref{eq:3barnorm-2}).
Combine the above, we have for all $\chi\in V_h$ and $\xi\in V_{0,h}$,
\begin{equation} \label{eq:infsup-biharmonic}
\sup_{\phi\in V_h,\, \psi\in V_{0,h}} \frac{\ldp \chi,\phi\rdp
-(\nabla_w \phi, \nabla_w \xi) - (\nabla_w \chi,\nabla_w
\psi)}{\|\phi\|_{0,h} + \3bar \psi\3bar}
 \gtrsim \|\chi\|_{0,h} + \3bar \xi\3bar.
\end{equation}

\medskip
\begin{theorem} \label{thm:errH1}
The weak Galerkin formulation (\ref{eq:wg}) for the biharmonic problem (\ref{pde}) has the following error estimate:
$$
\|\varepsilon_w\|_{0,h} + \3bar \varepsilon_u\3bar \lesssim h^m \|w\|_{m+1} + h^{n+\frac12} |\ln h| \| u\|_{W^{n+2,\infty}(\Omega)},
$$
where $\frac{1}{2}<m\le j+1$ and $0\le n\le j$.
Moreover, for $j\ge 1$, we have the improved estimate
  $$
  \|\varepsilon_w\|_{0,h} + \3bar \varepsilon_u\3bar \lesssim h^{m} \|w\|_{m+1} + h^{n+\frac12} \| u\|_{W^{n+2,\infty}(\Omega)},
  $$
where $\frac12< m\le j+1$ and $1\le n\le j$.
\end{theorem}

\begin{proof}
  By (\ref{eq:error}) and (\ref{eq:infsup-biharmonic}),
$$
  \begin{aligned}
    \|\varepsilon_w\|_{0,h}+ \3bar \varepsilon_u\3bar
    &\lesssim \sup_{\phi_h\in V_h,\, \psi_h\in V_{0,h}} \frac{\ldp \varepsilon_w,\phi_h\rdp -(\nabla_w \phi_h, \nabla_w \varepsilon_u)
                         - (\nabla_w \varepsilon_w,\nabla_w \psi_h)}{\|\phi_h\|_{0,h} + \3bar \psi_h\3bar} \\
    &= \sup_{\phi_h\in V_h,\, \psi_h\in V_{0,h}} \frac{E(w, u, \phi_h)}{\|\phi_h\|_{0,h} + \3bar \psi_h\3bar}.
  \end{aligned}
$$
Combining this with Lemma \ref{lem:E1E2}, this completes the proof of the theorem.
\end{proof}
\medskip

\begin{remark}
Assume that the exact solution $w$ and $u$ are sufficiently smooth.
It follows from the above theorem that the following convergence
holds true
$$
 \|\varepsilon_w\|_{0,h} + \3bar \varepsilon_u\3bar \lesssim \begin{cases}
O(h^{\frac12}|\ln h|)\quad &\textrm{for }j=0,\\
O(h^{j+\frac12})\quad &\textrm{for }j\ge 1.
\end{cases}
$$

\end{remark}

At this stage, it is standard to use the duality argument and derive
an error estimation for the $L^2$ norm of $\varepsilon_u$. However,
estimating $\|\varepsilon_u\|_{0,h}$ is not an easy task, as is
similar to the case of Poisson equations. For simplicity, we only
consider $\|\varepsilon_{u,0}\|$, where $\varepsilon_u$ is
conveniently expressed as $\varepsilon_u =
\{\varepsilon_{u,0},\varepsilon_{u,b}\}$. Define
\begin{equation} \label{eq:dualproblem}
\begin{cases}
\xi + \Delta \eta = 0, \\
-\Delta \xi = \varepsilon_{u,0},
\end{cases}
\end{equation}
where $\eta=0$ and $\frac{\partial \eta}{\partial \bn} = 0$ on $\partial \Omega$.
We assume that all internal angles of $\Omega$ are less than $126.283696\cdots^\circ$.
Then, according to (\ref{eq:regularity}), the solution to (\ref{eq:dualproblem}) has $H^4$ regularity:
$$
\|\xi\|_{2} + \|\eta\|_{4} \lesssim \|\varepsilon_{u,0}\|.
$$
Furthermore, since such a domain $\Omega$ is convex, the Poisson equation with either the homogeneous Dirichlet boundary condition
or the homogeneous Neumann boundary condition has $H^2$ regularity.

Clearly, Equation (\ref{eq:dualproblem}) can be written into the following form:
\begin{equation} \label{eq:dualweakform}
\begin{cases}
\ldp N_h\xi,\,\phi_h\rdp -(\nabla_w R_h\eta,\, \nabla_w\phi_h)=E(\xi,\eta,\phi_h) \; &\textrm{for all }\phi_h = \{\phi_0,\, \phi_b\}\in V_h,\\
(\nabla_w N_h\xi,\,\nabla_w\psi_h)=(\varepsilon_{u,0},\,\psi_0) \;
&\textrm{for all }\psi_h = \{\psi_0,\, \psi_b\}\in V_{0,h}.
\end{cases}
\end{equation}

For simplicity of the notation, denote
$$
\Lambda(N_h\xi, R_h\eta_h;\, \phi_h, \psi_h) = \ldp
N_h\xi,\,\phi_h\rdp -(\nabla_w R_h\eta,\, \nabla_w\phi_h) -
(\nabla_w N_h\xi,\,\nabla_w\psi_h).
$$
Note that $\Lambda$ is a symmetric bilinear form.
By setting $\phi_h = \varepsilon_w$ and $\psi_h = \varepsilon_u$ in (\ref{eq:dualweakform}) and then subtract these two equations,  one get
\begin{equation} \label{eq:EE}
\begin{aligned}
\|\varepsilon_{u,0}\|^2 &= E(\xi,\eta,\varepsilon_w) -\Lambda(N_h\xi, R_h\eta;\, \varepsilon_w, \varepsilon_u) \\
&= E(\xi,\eta,\varepsilon_w) -\Lambda(\varepsilon_w, \varepsilon_u;\, N_h\xi, R_h\eta) \\
&= E(\xi, \eta,\varepsilon_w) - E(w,u, N_h\xi) .
\end{aligned}
\end{equation}
Here we have used the symmetry of $\Lambda(\cdot,\cdot)$ and Equation (\ref{eq:error}).

The two terms, $E(\xi, \eta,\varepsilon_w)$ and $E(w,u, N_h\xi)$, in the right-hand side of Equation (\ref{eq:EE}) will be estimated one by one.
We start from $E(\xi, \eta,\varepsilon_w)$.
By using Lemma \ref{lem:E1E2}, it follows that
\begin{itemize}
\item[(i)] When $j=0$,
\begin{equation} \label{eq:EE1}
\begin{aligned}
E(\xi, \eta,\varepsilon_w) &\lesssim \left(h \|\xi\|_{2} + h^{\frac12} |\ln h| \|\eta\|_{W^{2,\infty}(\Omega)} \right) \|\varepsilon_w\|_{0,h} \\
&\lesssim h^{1/2}|\ln h|\left( \|\xi\|_{2} + \|\eta\|_{4}\right) \|\varepsilon_w\|_{0,h}.
\end{aligned}
\end{equation}
\item[(ii)] When $j\ge 1$, let $\delta>0$ be an infinitely small number which ensures
the Sobolev embedding from $W^{4,2}(\Omega)$ to $W^{3-\delta,\infty}(\Omega)$.
Then
\begin{equation}\label{eq:EE2}
\begin{aligned}
  E(\xi, \eta,\varepsilon_w) &\lesssim \left(h \|\xi\|_{2} + h^{\frac{3}{2}-\delta}|\ln h|  \|\eta\|_{W^{3-\delta,\infty}(\Omega)} \right) \|\varepsilon_w\|_{0,h}  \\
&\lesssim h\left( \|\xi\|_{2} + \|\eta\|_{4}\right) \|\varepsilon_w\|_{0,h}.
\end{aligned}
\end{equation}
\end{itemize}

Next, we give an estimate for $E(w,u, N_h\xi)$.
\medskip
\begin{lemma} \label{lem:EE3}
Assume all internal angles of $\Omega$ are less than $126.283696\cdots^\circ$,
which means the biharmonic problem with clamped boundary condition in $\Omega$ has $H^4$ regularity.
Then
\begin{itemize}
\item[(i)] For $j=0$,
$$
E(w,u, N_h\xi) \lesssim \left( h^{m+\frac12}\|w\|_{m+1} + h^2\|(I-Q_0)f\| + h^{n+1}\|u\|_{n+1} \right)\|\xi\|_2,
$$
where $\frac12<m\le 1$ and $1/2 < n\le 1$.
\item[(ii)] For $j\ge 1$,
  $$
  E(w,u, N_h\xi) \lesssim \left( h^{m+1}\|w\|_{m+1} + h^2\|(I-Q_0)f\| + h^{n+1}\|u\|_{n+1} \right)\|\xi\|_2,
  $$
where $\frac12< m\le j+1$ and $1/2 < n\le j+1$.
\end{itemize}
\end{lemma}
\medskip

\begin{proof}
  By definition,
\begin{equation} \label{eq:l2proof-1}
E(w,u, N_h\xi) = \ldp N_h w - w, N_h\xi\rdp + (\bPi_h\nabla u -
\nabla_w R_h u, \nabla_w N_h\xi).
\end{equation}
First, by the definition of $\ldp\cdot,\cdot\rdp$, the Schwarz inequality,
Remark \ref{rem:fullregularity-elliptic} and \ref{rem:ellipticL2error}, we have
\begin{equation}\label{eq:l2proof-2}
\begin{aligned}
  &\ldp N_h w - w, N_h\xi\rdp \\
  =\,&  (N_0 w - Q_0w, N_0\xi) + \sum_{K\in \T_h} h (N_0w-N_b w, N_0 \xi - N_b \xi)_{\partial K} \\
  \lesssim\,& \|N_0 w - Q_0w\| \|N_0 \xi\| +  \|N_hw-w\|_{0,h}\|N_h\xi-\xi\|_{0,h} \\
   \lesssim\,& \begin{cases}
    (h^{m+\frac12}\|w\|_{m+1} + h^2\|(I-Q_0)\Delta w\| )\|\xi\|_2 \;&\textrm{for }j=0,\;\frac12<m\le 1 \\
    (h^{m+1}\|w\|_{m+1} + h^2\|(I-Q_0)\Delta w\| )\|\xi\|_2 \;&\textrm{for }j\ge 1,\;\frac12<m\le j+1
    \end{cases}.
\end{aligned}
\end{equation}
Next, by using inequalities (\ref{commutative}), (\ref{eq:Neumann}), (\ref{div-q-homogeneous}),
(\ref{a2}), (\ref{eq:Neumann-error-H1})
and (\ref{eq:Ritz-error-L2}) one after one, we get
$$
\begin{aligned}
  &(\bPi_h\nabla u - \nabla_w R_h u, \nabla_w N_h\xi) \\
=\, & ((\bPi_h-\bP_h)\nabla u, \nabla_w N_h \xi) + (\nabla_w(Q_hu-R_h u), \nabla_w N_h\xi) \\
=\, & ((\bPi_h-\bP_h)\nabla u, \nabla_w N_h \xi) + (\nabla_w(Q_hu-R_h u), \bPi_h\nabla\xi) \\
=\, & ((\bPi_h-\bP_h)\nabla u, \nabla_w (N_h \xi-Q_h\xi)) + ((\bPi_h-\bP_h)\nabla u, \bP_h\nabla \xi) - (Q_0u-R_0 u, \Delta\xi) \\
\lesssim \, & h^{n+1}\|u\|_{n+1}\|\xi\|_2 + ((\bPi_h-I)\nabla u, \bP_h\nabla \xi)
              + h^2\|(I-Q_0)\Delta u\| \|\xi\|_2,
\end{aligned}
$$
for $\frac12< n\le j+1$.
The estimation for $((\bPi_h-I)\nabla u, \bP_h\nabla \xi)$
follows the same technique used in Inequality (\ref{eq:duality-component}).
By the definition of $\bPi_h$ and since $\frac{\partial u}{\partial\bn}=0$ on $\partial \Omega$,
we know that $(\bPi_h-I)\nabla u\cdot\vn $ also vanishes on $\partial\Omega$.
Therefore, using the same argument as in (\ref{eq:duality-component}), one has
$$
\begin{aligned}
((\bPi_h-I)\nabla u, \bP_h\nabla \xi)&\lesssim h^{n+1}\|u\|_{n+1} \|\xi\|_2 + h^2 \|(I-Q_0)\Delta u\| \|\xi\|_2 
\end{aligned}
$$
for $\frac12< n\le j+1$. Combining the above gives
\begin{equation}\label{eq:l2proof-3}
(\bPi_h\nabla u - \nabla_w R_h u, \nabla_w N_h\xi) \lesssim
\left(h^{n+1}\|u\|_{n+1} + h^2\|(I-Q_0)\Delta u\|\right) \|\xi\|_2.
\end{equation}
for $\frac12< n\le j+1$.

Notice that
\begin{equation}\label{eq:l2proof-4}
\begin{aligned}
h^2\|(I-Q_0)\Delta u\| &= h^2\|(I-Q_0) w\| \lesssim h^{m+2}\|w\|_{m}\qquad\textrm{for }0\le m\le j+1,\\
h^2\|(I-Q_0)\Delta w\| &= h^2\|(I-Q_0) f\|.
\end{aligned}
\end{equation}
The lemma follows immediately from (\ref{eq:l2proof-1})-(\ref{eq:l2proof-4}).
\end{proof}
\medskip

Finally, combining Theorem \ref{thm:errH1}, inequalities (\ref{eq:EE}), (\ref{eq:EE1})-(\ref{eq:EE2}), and Lemma \ref{lem:EE3},
we get the following $L^2$ error estimation:
\medskip
\begin{theorem} \label{thm:errL2}
Assume all internal angles of $\Omega$ are less than $126.283696\cdots^\circ$,
which means the biharmonic problem with clamped boundary condition in $\Omega$ has $H^4$ regularity. Then
\begin{itemize}
\item[(i)]  For $j=0$,
$$
\begin{aligned}
\|\varepsilon_{u,0}\| &\lesssim h^{m+\frac12}|\ln h| \|w\|_{m+1} + h|\ln h|^2\| u\|_{W^{2,\infty}(\Omega)} \\
 &\qquad\qquad + h^2\|(I-Q_0)f\| + h^{n+1}\|u\|_{n+1},
\end{aligned}
$$
where $\frac12 < m \le 1$ and $\frac12 < n \le 1$.
\item[(ii)]  For $j\ge 1$,
  $$
  \|\varepsilon_{u,0}\| \lesssim h^{m+1}\|w\|_{m+1} + h^{l+\frac{3}{2}}\|u\|_{W^{l+2,\infty}(\Omega)} + h^2\|(I-Q_0)f\| + h^{n+1}\|u\|_{n+1},
  $$
  where $\frac12< m\le j+1$, $\frac12< n\le j+1$ and $1\le l\le j$.
\end{itemize}
\end{theorem}

\medskip

\begin{remark}
If $u$, $w$ and $f$ are sufficiently smooth, then we get
$$
\|\varepsilon_{u,0}\| \lesssim
\begin{cases}
  O(h|\ln h|^2) \quad &\textrm{for }j=0, \\
  O(h^{j+\frac32})   \quad &\textrm{for }j\ge 1.
\end{cases}
$$
\end{remark}

\section{Numerical results} \label{sec:numerical}

In this section, we would like to report some numerical results for
the weak Galerkin finite element method proposed and analyzed in
previous sections.
Before doing that, let us briefly review some existing results for
$H^1$-$H^1$ conforming, equal-order finite element discretization of
the Ciarlet-Raviart mixed formulation. As discussed in
\cite{Babuska80, Scholz78}, theoretical error estimates for such
schemes are indeed sub-optimal due to an effect of
$\inf_{\chi_h}\|u-\chi_h\|_{2}$, where $\chi_h$ is taken from the
employed $H^1$ conforming finite element space. For example, when
$H^1$-$H^1$ conforming quadratic elements are used to approximate
both $u$ and $w$, the error satisfies $\|u-u_h\|_{2} + \|w-w_h\|
\lesssim \inf_{\chi_h}\|u-\chi_h\|_{2} + \inf_{\chi_h}\|w-\chi_h\|
\lesssim O(h)$, while intuitively, one may expect $\|w-w_h\|$ to
have an $O(h^2)$ convergence. By using the $L^{\infty}$ argument,
Scholz \cite{Scholz78} was able to improve the convergence rate of
$L^2$ norm for $w$ by $h^{\frac12}$, and it is known that this
theoretical result is indeed sharp. For the weak Galerkin
approximation, from the discussing in the previous sections, clearly
we are facing the same issue.

However, numerous numerical experiments have illustrated that
$H^1$-$H^1$ conforming, equal-order Ciarlet-Raviart mixed finite
element approximation often demonstrates convergence rates better
than the theoretical prediction. Indeed, this has been partly
explained theoretically in \cite{Scholz79}, in which the author
proved that optimal order of convergence rates can be recovered in
certain fixed subdomains of $\Omega$, when equal order $H^1$
conforming elements are used. We point out that similar phenomena
have been observed in the numerical experiments using weak Galerkin
discretization. This means that numerical results are often better
than theoretical predictions.

Another issue in the implementation of the weak Galerkin finite
element method is the treatment of non-homogeneous boundary data
$$
\begin{aligned}
  u&=g_1\qquad \mbox{on}\; \partial\Omega,\\
  \frac{\partial u}{\partial\bn}&=g_2 \qquad \mbox{on}\; \partial\Omega.
\end{aligned}
$$
Clearly, both boundary conditions are imposed on $u$, and $u=g_1$ is
the essential boundary condition while $\frac{\partial
u}{\partial\bn}=g_2$ is the natural boundary condition. To impose
the natural boundary condition, we shall modify the first equation
of (\ref{eq:wg}) into
$$
\ldp w_h,\,\phi_h\rdp -(\nabla_w u_h,\, \nabla_w\phi_h)= -\langle
g_2, \phi_b\rangle_{\partial\Omega}.
$$
The essential boundary condition should be enforced by taking the
$L^2$ projection of the corresponding boundary data.

Consider three test problems defined on $\Omega=[0,1]\times[0,1]$
with exact solutions
$$
\begin{aligned}
u_1 &= x^2(1-x)^2y^2(1-y)^2,\\
u_2 &= \sin(2\pi x)\sin(2\pi y)\qquad\textrm{and}\qquad u_3 =
\sin(2\pi x+\frac{\pi}{2})\sin(2\pi y+\frac{\pi}{2}),
\end{aligned}
$$
respectively. The reason for choosing these three exact solutions is
that they have the following type of boundary conditions
$$
\begin{aligned}
u_1|_{\partial\Omega} &= 0 \qquad & \left.\frac{\partial u_1}{\partial \bn}\right|_{\partial\Omega} &= 0, \\
u_2|_{\partial\Omega} &= 0 \qquad & \left.\frac{\partial u_2}{\partial \bn}\right|_{\partial\Omega} &\neq 0, \\
u_3|_{\partial\Omega} &\neq 0 \qquad & \left.\frac{\partial u_3}{\partial \bn}\right|_{\partial\Omega} &= 0. \\
\end{aligned}
$$
This allows us to test the effect of different boundary data on convergence rates.
Although the theoretical error estimates are given for $\varepsilon_u=R_hu-u_h$
and $\varepsilon_w=N_hw-w_h$, it is clear that they have at least the same order
as $e_u=Q_h u - u_h$ and $e_w = N_hw-w_h$, provided that the exact
solution is smooth enough. Thus for convenience, we only compute different
norms for $e_u$ and $e_w$, instead of for $\varepsilon_u$ and $\varepsilon_w$.

The tests are performed using an unstructured triangular initial
mesh, with characteristic mesh size $0.1$. The initial mesh is then
refined by dividing every triangle into four sub-triangles, to
generate a sequence of nested meshes with various mesh size $h$. All
discretization schemes are formulated by using the lowest order weak
Galerkin element, with $j=0$. For simplicity of notation, for any
$v\in V_h$, denote
$$
\|v_b\| = \left(\sum_{K\in \mathcal{T}_h} h\|v_b\|_{\partial K}^2
\right)^{1/2}.
$$

The results for test problems with exact solutions $u_1$, $u_2$ and
$u_3$, are reported in Table \ref{tab:prob1}, \ref{tab:prob2} and
\ref{tab:prob3}, respectively. The results indicate that $u$ always
achieves an optimal order of convergence, while the convergence for
$w$ varies with different boundary conditions. It should be pointed
out that both of them have outperformed the convergence as predicted
by theory.

\begin{table}[h]
  \caption{Numerical results for the test problem with exact solution $u_1$ and lowest order of WG
  elements.} \label{tab:prob1}
  \begin{center}
    \begin{tabular}{|c|c|c|c|c|c|c|}
      \hline
     $h$  & $\|\nabla_w e_u\|$ & $\|e_{u,0}\|$ & $\|e_{u,b}\|$ & $\|\nabla_w e_w\|$ & $\|e_{w,0}\|$ & $\|e_{w,b}\|$ \\ \hline
   0.1     &  1.33e-03 &  2.40e-04 &  4.59e-04 &  5.66e-02 &  2.96e-03 &  6.91e-03 \\ \hline
   0.05    &  4.69e-04 &  6.18e-05 &  1.17e-04 &  2.80e-02 &  9.14e-04 &  1.99e-03 \\ \hline
   0.025   &  2.00e-04 &  1.55e-05 &  2.97e-05 &  1.60e-02 &  2.64e-04 &  5.70e-04 \\ \hline
   0.0125  &  9.56e-05 &  3.90e-06 &  7.44e-06 &  1.21e-02 &  8.33e-05 &  1.89e-04 \\ \hline
   0.00625 &  4.72e-05 &  9.77e-07 &  1.86e-06 &  1.13e-02 &  3.26e-05 &  7.91e-05 \\ \hline
   $\begin{matrix}\textrm{Asym. Order} \\ O(h^k),\; k= \end{matrix}$&
    1.1930  &  1.9876  &  1.9877 &  0.5864  &  1.6461  &  1.6298 \\ \hline
    \end{tabular}
  \end{center}
\end{table}

\begin{table}[h]
  \caption{Numerical results for the test problem with exact solution $u_2$ and lowest order of WG
  elements.} \label{tab:prob2}
  \begin{center}
    \begin{tabular}{|c|c|c|c|c|c|c|}
      \hline
     $h$  & $\|\nabla_w e_u\|$ & $\|e_{u,0}\|$ & $\|e_{u,b}\|$ & $\|\nabla_w e_w\|$ & $\|e_{w,0}\|$ & $\|e_{w,b}\|$ \\ \hline
  0.1     &  9.58e-01 &  8.66e-02 &  1.65e-01 &  4.39e+01 &  6.09e-01 &  2.01e+00 \\ \hline
  0.05    &  3.34e-01 &  2.18e-02 &  4.14e-02 &  2.32e+01 &  2.78e-01 &  7.19e-01 \\ \hline
  0.025   &  1.43e-01 &  5.47e-03 &  1.03e-02 &  1.37e+01 &  1.15e-01 &  2.81e-01 \\ \hline
  0.0125  &  6.81e-02 &  1.37e-03 &  2.59e-03 &  1.02e+01 &  5.12e-02 &  1.26e-01 \\ \hline
  0.00625 &  3.36e-02 &  3.42e-04 &  6.49e-04 &  9.33e+00 &  2.45e-02 &  6.12e-02 \\ \hline
  $\begin{matrix}\textrm{Asym. Order} \\ O(h^k),\; k= \end{matrix}$&
    1.1958  &  1.9958  &  1.9975 &  0.5649  &  1.1709  &  1.2587 \\ \hline
    \end{tabular}
  \end{center}
\end{table}

\begin{table}[h]
  \caption{Numerical results for the test problem with exact solution $u_3$ and lowest order of WG
  elements.} \label{tab:prob3}
  \begin{center}
    \begin{tabular}{|c|c|c|c|c|c|c|}
      \hline
      $h$ & $\|\nabla_w e_u\|$ & $\|e_{u,0}\|$ & $\|e_{u,b}\|$ & $\|\nabla_w e_w\|$ & $\|e_{w,0}\|$ & $\|e_{w,b}\|$ \\ \hline
   0.1     &  8.23e-01 &  1.18e-01 &  2.27e-01 &  5.61e+01 &  4.25e+00 &  9.42e+00 \\ \hline
   0.05    &  3.07e-01 &  3.18e-02 &  6.09e-02 &  2.43e+01 &  1.24e+00 &  2.58e+00 \\ \hline
   0.025   &  1.35e-01 &  8.13e-03 &  1.55e-02 &  1.13e+01 &  3.28e-01 &  6.61e-01 \\ \hline
   0.0125  &  6.49e-02 &  2.04e-03 &  3.90e-03 &  5.58e+00 &  8.42e-02 &  1.67e-01 \\ \hline
   0.00625 &  3.21e-02 &  5.11e-04 &  9.78e-04 &  2.77e+00 &  2.14e-02 &  4.21e-02 \\ \hline
   $\begin{matrix}\textrm{Asym. Order} \\ O(h^k),\; k= \end{matrix}$&
   1.1599  &  1.9679  &  1.9682  & 1.0801  &  1.9157  &  1.9558   \\ \hline
    \end{tabular}
  \end{center}
\end{table}

Our final example is a case where the exact solution has a low
regularity in the domain $\Omega=[0,1]\times[0,1]$. More precisely,
the exact solution is given by
$$
u_4 = r^{3/2}\left(\sin\frac{3\theta}{2} -
3\sin\frac{\theta}{2}\right),
$$
where $(r, \theta)$ are the polar coordinates. It is easy to check
that $u\in H^{2.5}$. The errors for weak Galerkin finite element
approximations are reported in Table \ref{tab:prob4}. Here, $u$
still achieves an optimal order of convergence, while the
convergence rates for $w$ is restricted by the fact that $w\in
H^{0.5}$. All the results are in consistency with the theory
established in this article.

\begin{table}[h]
  \caption{Numerical results for the test problem with exact solution $u_4$ and lowest order of WG
  elements.} \label{tab:prob4}
  \begin{center}
    \begin{tabular}{|c|c|c|c|c|c|c|}
      \hline
      $h$ & $\|\nabla_w e_u\|$ & $\|e_{u,0}\|$ & $\|e_{u,b}\|$ & $\|\nabla_w e_w\|$ & $\|e_{w,0}\|$ & $\|e_{w,b}\|$ \\ \hline
   0.1     &  3.73e-02 &  9.44e-04 &  2.15e-03 &  2.88e+01 &  4.05e-01 &  1.78e+00 \\ \hline
   0.05    &  1.87e-02 &  2.55e-04 &  5.73e-04 &  4.08e+01 &  2.86e-01 &  1.26e+00 \\ \hline
   0.025   &  9.37e-03 &  6.60e-05 &  1.46e-04 &  5.77e+01 &  2.02e-01 &  8.91e-01 \\ \hline
   0.0125  &  4.68e-03 &  1.67e-05 &  3.69e-05 &  8.16e+01 &  1.42e-01 &  6.30e-01 \\ \hline
   0.00625 &  2.34e-03 &  4.19e-06 &  9.24e-06 &  1.15e+02 &  1.01e-01 &  4.45e-01 \\ \hline
   $\begin{matrix}\textrm{Asym. Order} \\ O(h^k),\; k= \end{matrix}$&
    0.9984 &   1.9567 &   1.9690 &  -0.4998 &   0.5008 &   0.5000 \\ \hline
   \end{tabular}
  \end{center}
\end{table}

\end{document}